\documentclass{amsart}

\usepackage{paralist}
\usepackage{url}
\usepackage{hyperref}
\usepackage{color}
\usepackage{mathrsfs}
\usepackage{graphicx}

\newcommand{\RR}{{\mathbb R}}
\newcommand{\CC}{{\mathbb C}}
\newcommand{\re}{\mathbb{R}}
\newcommand{\mR}{\mathbb{R}}

\newcommand{\mC}{\mathbb{C}}

\newcommand{\N}{\mathbb{N}}

\def\af{\alpha}

\newcommand{\bdes}{\begin{description}}
\newcommand{\edes}{\end{description}}
\newcommand{\bal}{\begin{align}}
\newcommand{\eal}{\end{align}}
\newcommand{\bnum}{\begin{enumerate}}
\newcommand{\enum}{\end{enumerate}}
\newcommand{\bit}{\begin{itemize}}
\newcommand{\eit}{\end{itemize}}
\newcommand{\bea}{\begin{eqnarray}}
\newcommand{\eea}{\end{eqnarray}}
\newcommand{\be}{\begin{equation}}
\newcommand{\ee}{\end{equation}}
\newcommand{\baray}{\begin{array}}
\newcommand{\earay}{\end{array}}
\newcommand{\bsry}{\begin{subarray}}
\newcommand{\esry}{\end{subarray}}
\newcommand{\bca}{\begin{cases}}
\newcommand{\eca}{\end{cases}}
\newcommand{\bcen}{\begin{center}}
\newcommand{\ecen}{\end{center}}
\newcommand{\bbm}{\begin{bmatrix}}
\newcommand{\ebm}{\end{bmatrix}}
\newcommand{\bmx}{\begin{matrix}}
\newcommand{\emx}{\end{matrix}}
\newcommand{\bpm}{\begin{pmatrix}}
\newcommand{\epm}{\end{pmatrix}}
\newcommand{\btab}{\begin{tabular}}
\newcommand{\etab}{\end{tabular}}

\newtheorem{theorem}{Theorem}[section]

\newtheorem{prop}[theorem]{Proposition}

\newtheorem{lemma}[theorem]{Lemma}
\newtheorem{cor}[theorem]{Corollary}

\newtheorem{defi}[theorem]{Definition}
\theoremstyle{definition}
\newtheorem{example}[theorem]{Example}
\newtheorem{exm}[theorem]{Example}
\newtheorem{alg}[theorem]{Algorithm}

\newtheorem{remark}[theorem]{Remark}
\newtheorem{define}[theorem]{Definition}

\numberwithin{equation}{section}

\begin{document}
% Li change: March 8 the title
% Li change: march 23
\title[Semidefinite relaxations for semi-infinite polynomial programming]
{Semidefinite Relaxations for semi-infinite polynomial programming}
\author{Li Wang}
\address{Department of Mathematics,
University of California, 9500 Gilman Drive, La Jolla, CA 92093.}
\email{liw022@ucsd.edu}
\author{Feng Guo}
\address{Department of
Mathematics, University of California, 9500 Gilman Drive, La Jolla,
CA 92093.} \email{f1guo@math.ucsd.edu.}

\begin{abstract}
This paper studies how to solve semi-infinite polynomial programming
(SIPP) problems by semidefinite relaxation method. We first
introduce two SDP relaxation methods for solving polynomial
optimization problems with finitely many constraints. Then we
propose an exchange algorithm with SDP relaxations to solve SIPP
problems with compact index set. At last, we extend the proposed
method to
%solve
SIPP problems with noncompact index set via homogenization.
Numerical results show that the algorithm is efficient in practice.
\end{abstract}
\maketitle

\section{Introduction}
Consider the {\itshape semi-infinite polynomial programming} (SIPP)
problem:
\begin{equation*}\label{semiinfinite:pp}
(P): \left\{
\begin{aligned}
f^* := \min\limits_{x\in X}&\ f(x) \\
\text{s.t.}&\ g(x,u)\geq0, \ \forall~u\in U,
\end{aligned}
\right.
\end{equation*}
where
\begin{equation*}
\begin{aligned}
X&=\{x\in\RR^n\mid \theta_1(x)\geq 0,\cdots,\theta_{m_2}(x)\geq 0\},\\
U&=\{u\in\RR^p\mid h_1(u)\geq 0,\cdots,h_{m_1}(u)\geq 0\}.\\
\end{aligned}
\end{equation*}
Here $f(x), \theta_i(x)$ are polynomials in $x\in \mathbb{R}^n$,
$h_j(u)$ are polynomials in $u\in \mathbb{R}^p$ and $g(x,u)$ is a
polynomial in $(x,u) \in \mathbb{R}^n\times \mathbb{R}^p$.
Throughout this paper, we assume that $X$  is compact and $U$ is an
infinite index set, i.e., there are infinitely many constraints in
$(P)$. The SIPP problem is a special subclass of the {\itshape
semi-infinite programming} (SIP) which has many applications, e.g.,
Chebyshev approximation, maneuverability problems, some mathematical
physics problems and so on \cite{Hettich1993,Still2007}.

There are various algorithms for SIP problems based on
discretization schemes of $U$, such as central cutting plane method
\cite{cuttingplane}, Newton's method \cite{Still}, SQP methods
\cite{SQPSIP} and the like. Most of algorithms for SIP problems,
however, are only locally convergent or globally convergent under
some strong assumptions, like convexity or linearity, and to the
authors best knowledge, few of them are specially designed for SIPP
problems exploiting features of polynomial optimization problems.
Parpas and Rustem \cite{PR2009} proposed a discretization like
method to solve min-max polynomial optimization problems, which can
be reformulated as SIPP problems. Using a polynomial approximation
and an appropriate hierarchy of semidefinite relaxations, Lasserre
presented an algorithm to solve the {\itshape generalized} SIPP
problems in \cite{Lasserre2011}.

Before introducing the contribution of this paper, we first review
some of the considerable progress recently made in solving
polynomial optimization problems with {\itshape finite} constraints
via sums of squares relaxations, which are typically based on the
Positivstellensatz \cite{Putinar1993}. We define a so-called
quadratic module which is a set of polynomials generated by the
finitely many constraints, to which any polynomials positive over
the feasible set belong. The classic Lasserre's hierarchy
\cite{Las01} is to compute the maximal real number, minus which the
objective lies in the quadratic module. By increasing the order of
the quadratic module, Lasserre's hierarchy results in a sequence of
lower bounds of the global optimum and the asymptotical convergence
is established under the Archimedean Condition. Interestingly,
finite convergence of Lasserre's hierarchy is generic
\cite{NiefiniteLas}. To guarantee the finite convergence of
Lasserre's hierarchy, Nie \cite{Nie2011Jacobian} proposed a refined
SDP relaxation by some ``Jacobian-type'' technique which represents
optimality conditions of the considered polynomial optimization
problem. More importantly, these SDP relaxation methods are global
and the minimizers can be extracted if the flat extension condition
\cite{Curto2005} or more general, flat truncation condition
\cite{Nie2011certify} holds. The aim of this paper is to apply these
SDP relaxation methods to solve SIPP problems.

An efficient method based on discretization scheme for solving SIP
is the exchange method which approaches the optimum in an iterative
manner. Generally speaking, given a finite subset $U_k\subseteq U$
in an iteration, we obtain at least one global minimizer $x^k$ of
$f(x)$ under the associated finitely many constraints and then
compute the global minimum $g^k$ and minimizers $u_1,\ldots,u_t$ of
$g(x^k,u)$ over $U$. If $g^k\ge 0$, stop; otherwise, update
$U_{k+1}=U_k\cup\{u_1,\ldots,u_t\}$ and proceed to the next
iteration. Therefore, to guarantee the success of the exchange
method, the subproblems in each iteration need to be globally solved
and at least one minimizer of each subproblem can be extracted. The
compactness of the index set $U$ is commonly assumed in many
algorithms for SIP problems, which ensure the existence of global
minimizers for constraint subproblem. However, when the constraint
subproblem is nonconvex, globally solving it and extracting global
minimizers are very challenging.

Specializing the exchange method in SIPP problem $(P)$, the
subproblems are polynomial optimization problems with finitely many
constraints, which can be solved exactly by SDP relaxations.
Assuming the index set $U$ is compact, an exchange type method with
SDP relaxations is given in this paper. Numerical experiments show
that this algorithm is efficient in practice. We also apply this
approach to optimization problems with polynomial matrix inequality
and get good numerical performance. If $U$ is noncompact, the
exchange method might fail, see Example \ref{ex::count}. Another
novelty of this paper is that we extend the proposed algorithm to
solve SIPP problems with noncompact $U$. By a technique of
homogenization, we first reformulate the original SIPP problem as a
new one with a compact index set, to which we then apply the
proposed semidefinite relaxation algorithm.  We prove that these two
problems are equivalent under some generic conditions.

The paper is organized as follows. In Section \ref{sec::SDP}, we
introduce two SDP relaxation methods  for solving polynomial
optimization problems with finitely many constraints. In Section
\ref{sec::compact}, we propose a semidefinite relaxation algorithm
to solve SIPP problem $(P)$ with compact index set $U$. In Section
\ref{sec::noncompact}, we consider how to apply the proposed
algorithm to solve SIPP problems with noncompact index set $U$ by
homogenization.

\vspace{8pt} \noindent {\bf Notation.} The symbol $\N$ (resp.,
$\re$, $\mathbb{C}$) denotes the set of nonnegative integers (resp.,
real numbers, complex numbers). For any $t\in \re$, $\lceil t\rceil$
% Li change Feb 8 add "that is"
denotes the smallest integer that is not smaller than $t$. For
% Li change March 4, definition of [n]_k
integer $n>0$, $[n]$ denotes the set $\{1,\cdots,n\}$. For $x \in
\re^n$, $x_i$ denotes the $i$-th component of $x$. For $x \in \re^n$
and $\af \in \N^n$, $x^\af$ denotes $x_1^{\af_1}\cdots x_n^{\af_n}$.
For a finite set $T$, $|T|$ denotes its cardinality. $\mathbb{R}[x]
= \mathbb{R}[x_1,\cdots,x_n]$ denotes the ring of polynomials in
$(x_1,\cdots,x_n)$ with real coefficients.
 For a symmetric matrix
$W$, $W\succeq 0(\succ 0)$ means that $W$ is positive semidefinite
(definite). For any vector $u\in \re^p$, $\| u \|$ denotes the
standard Euclidean 2-norm.

\section{SDP relaxations for polynomial optimization}\label{sec::SDP}
In this section, we study how to solve the following polynomial
optimization problem with finitely many constraints:
\begin{equation}\label{pro::opti}
\left\{
\begin{aligned}
f_{\min}:=\underset{x\in\RR^n}{\min}&\ f(x)\\
\text{s.t.}&\ h_1(x)=\cdots=h_{m_1}(x)=0,\\
&\ g_1(x)\ge 0, \ldots, g_{m_2}(x)\ge 0,
\end{aligned}\right.
\end{equation}
where $f(x),h_i(x),g_j(x)\in\RR[x]$. Based on the
% Li change: has to have
Positivstellensatz, considerable works have recently been done on
% Li change: relaxation
solving (\ref{pro::opti}) by means of SDP relaxation. Generally
speaking, these methods relax (\ref{pro::opti}) as a sequence of
SDPs whose optima are lower bounds of $f_{\min}$ and converge to
$f_{\min}$ under some assumptions. We first introduce the classic
% Li change: Lasserre's hierarchy
Lasserre's SDP relaxation \cite{Las01} and then Nie's Jacobian SDP
relaxation \cite{Nie2011Jacobian} with property of finite
convergence.

\subsection{Lasserre's SDP relaxation}
% Li change: Feb 10 "to be " to "as"
Denote $K$ as the feasible set of \eqref{pro::opti}. Let
$\mathcal{F}:=\{h_1,\ldots,h_{m_1},g_0,g_1,\ldots,g_{m_2}\}$ and
$g_0=1$. We say a polynomial is SOS if it is a sum of squares of
other polynomials. The $k$-th truncated {\itshape quadratic module}
generated by $\mathcal{F}$ is defined as
\[
Q_k(\mathcal{F}):=\left\{\sum\limits_{j=1}^{m_1}\phi_jh_j+\sum\limits_{i=0}^{m_2}\sigma_ig_i\Bigg|
\begin{aligned}
% Li change March 9, add "~"
&\sigma_i\ \text{are SOS},\ \phi_j\in\RR[x],\ \forall~i,j\\
&\deg(\sigma_ig_i)\le 2k,\ \deg(\phi_jh_j)\le 2k
\end{aligned}
\right\}.
\]
The $k$-th Lasserre's SDP relaxation \cite{Las01} for solving
(\ref{pro::opti}) ($k$ is also called the relaxation order) is
\begin{equation}\label{eq::PR}
f_k:=\max~\gamma\quad \text{s.t.}\ f(x)-\gamma\in Q_k(\mathcal{F}).
\end{equation}
The relaxation (\ref{eq::PR}) is equivalent to a semidefinite
% Li change March 8, delete SDP
program and could be solved efficiently by numerical methods like
interior-point algorithms. Clearly, $f_k\le f_{\min}$ for every $k$
and the sequence $\{f_k\}$ is monotonically increasing.  The
quadratic module generated by $\mathcal{F}$ is
\[
Q(\mathcal{F}):=\bigcup_{k=1}^{\infty}Q_k(\mathcal{F}).
\]
\begin{define}\label{def::AC}
The set $Q(\mathcal{F})$ satisfies the {\itshape Archimedean
Condition} if there exists $\psi\in Q(\mathcal{F})$ such that
inequality $\psi(x)\ge 0$ defines a compact set in $x\in \mR^n$.
\end{define}
Note that the Archimedean Condition implies the feasible set $K$ is
compact but the inverse is not necessarily true. However, for any
compact $K$ we can always ``force'' the associated quadratic module
to satisfy the Archimedean Condition by adding a ``redundant''
constraint, e.g., $\rho-\Vert x\Vert^2\geq 0$ for sufficiently large
$\rho$.

The convergence for Lasserre's hierarchy (\ref{eq::PR}), i.e.,
$\lim_{k\rightarrow\infty}f_k=f_{\min}$, is implied by
Putinar's Positivstellensatz:
\begin{theorem}\emph{(}\cite{PR2009}\emph{)}\label{th::PP}
If a polynomial $p$ is positive on $K$ and the Archimedean Condition
holds, then $p\in Q(\mathcal{F})$.
\end{theorem}

We next consider the dual optimization problem of (\ref{eq::PR}).
Let $y$ be a {\itshape truncated moment sequence $($tms$)$} of
degree $2k$, i.e., $y=(y_{\alpha})$ be a sequence of real numbers
which are indexed by
$\alpha:=(\alpha_1,\ldots,\alpha_n)\in\mathbb{N}^n$ with
$|\alpha|:=\alpha_1+\cdots+\alpha_n\le 2k$.  The associated $k$-th
{\itshape moment matrix} is denoted as $M_k(y)$ which is indexed by
$\mathbb{N}^n_{k}$, with $(\alpha,\beta)$-th entry
$y_{\alpha+\beta}$.  Given polynomial
$p(x)=\sum_{\alpha}p_{\alpha}x^{\alpha}$ where
$x^{\alpha}:=x_1^{\alpha_1}\cdots x_n^{\alpha_n}$, denote
$d_p=\lceil \deg(p)/2\rceil$. For $k\ge d_p$, the $(k-d_p)$-th
{\itshape localizing moment matrix} $L_p^{(k-d_p)}(y)$ is defined as
the moment matrix of the {\itshape shifted vector}
$((py)_{\alpha})_{\alpha\in\mathbb{N}^n_{2(k-d_p)}}$ with
$(py)_{\alpha}=\sum_{\beta}p_{\beta}y_{\alpha+\beta}$. Denote by
${\mathscr{M}_{2k}}$ the space of all tms whose degrees are $2k$.
Let $\RR[x]_{2k}$ be the space of real polynomials in $x$ with
degree at most $2k$. For any $y\in\mathscr{M}_{2k}$, a Riesz
functional $\mathscr{L}_{y}$ on $\RR[x]_{2k}$ is defined as
\[
\mathscr{L}_y\left(\sum_{\alpha}q_{\alpha}x_1^{\alpha_1}\cdots
x_n^{\alpha_n}\right) =\sum_{\alpha}q_{\alpha}y_{\alpha},\quad
\forall~q(x)\in\RR[x]_{2k}.
\]
For convenience, we hereafter still use $q$ to denote the
coefficient vector of $q(x)$ in the graded lexicographical ordering
and denote $\langle q,y\rangle=\mathscr{L}_y(q)$.  From the
definition of the localizing moment matrix $L_p^{(k-d_p)}(y)$, it is
easy to check that
\[
q^TL_p^{(k-d_p)}(y)q=\mathscr{L}_y(p(x)q(x)^2),\quad \forall~
q(x)\in\RR[x]_{k-d_p}.
\]
The dual optimization problem of (\ref{eq::PR}) is
(\cite{Las01,Las09})
\begin{equation}\label{eq::DPR}
\left\{
\begin{aligned}
f_k^*:=\underset{y\in\mathscr{M}_{2k}}{\min}&\ \langle f,y\rangle \\
\text{s.t.}&\ L_{h_j}^{(k-d_{h_j})}(y)=0,\ j\in[m_1],\ L_{g_i}^{(k-d_{g_i})}(y)\succeq 0,\ i\in[m_2],\\
&\ M_k(y)\succeq 0,\ \langle 1,y\rangle=1.
\end{aligned}
\right.
\end{equation}
Let $$d=\max\{1,d_{g_i},d_{h_j}\mid i\in[m_1],j\in[m_2]\}.$$
Lasserre \cite{Las01} shows that $f_k\le f^*_k\le f_{\min}$ for
every $k\ge\max\{d_f,d\}$ and both $\{f_k\}$ and $\{f^*_k\}$
converge to $f_{\min}$ if the Archimedean Condition holds.

We say Lasserre's hierarchy (\ref{eq::PR}) and (\ref{eq::DPR}) has
{\itshape finite convergence} if
\begin{equation}\label{def::fc}
f_{k_1}=f^*_{k_1}=f_{\min}\quad\text{for some order}\ k_1<\infty.
\end{equation}

Interestingly, Nie proved that under the Archimedean Condition, Lasserre's SDP
relaxation has finite convergence {\itshape generically} (cf. \cite[Theorem
1.1]{NiefiniteLas}).  Since $f_{\min}$ is usually unknown, a practical issue is
how  to certify the finite convergence if it happens. Moreover, if it is
certified, how do we get minimizers?

Let $y^*$ be an optimizer of (\ref{eq::DPR}). By \cite[Theorem
1.1]{Curto2005}, $f_k^*=f_{\min}$ for some $k$ if the {\itshape flat
extension condition} (FEC) \cite{Curto2005} holds, i.e.,
\begin{equation}\label{eq::FEC}
\text{rank}~M_{k-d}(y^*)=\text{rank}~M_k(y^*).
\end{equation}
By solving some SVD and eigenvalue problems (\cite{HL05}), we can
get $r :=\text{rank}~M_k(y^*)$ global optimizers for
(\ref{pro::opti}). However, (\ref{eq::FEC}) is not a generally
necessary condition for checking finite convergence of Lasserre's
hierarchy (cf. \cite[Example 1.1]{Nie2011certify}). To certify the
finite convergence of (\ref{eq::PR}) and get minimizers of
(\ref{pro::opti}) from (\ref{eq::DPR}), a weaker condition was
proposed in \cite{Nie2011certify}. We say a minimizer $y^*$ of
(\ref{eq::DPR}) satisfies {\itshape flat truncation condition} (FTC)
if there exists an integer $t\in[\max\{d_f,d\},k]$ such that
\begin{equation}\label{eq::FTC}
\text{rank}~M_{t-d}(y^*)=\text{rank}~M_t(y^*).
\end{equation}
If an optimizer of (\ref{eq::DPR}) has a flat truncation, by
\cite[Theorem 1.1]{Curto2005} again, we still have $f_k^*=f_{\min}$.

Moreover, if there is no duality gap between (\ref{eq::PR}) and
(\ref{eq::DPR}), we obtain $f_k=f_{\min}$. More importantly,
\cite[Theorem 2.2]{Nie2011certify} shows that the flat truncation is
also necessary for Lasserre's hierarchy (\ref{eq::PR}) under some
{\itshape generic} assumptions.

\begin{alg}\label{alg:finit:cons:putinar}Lasserre's SDP relaxation

\noindent\textbf{Input:} Objective function $f(x)$, constraint
functions $h_i(x), g_j(x)$ and maximal relaxation order $k_{\max}$.

\noindent\textbf{Output:} Global minimum and minimizers of problem
(\ref{pro::opti}).
\begin{enumerate}[\upshape I]
\item Set $d:=\max\{1, d_f,d_{h_i},d_{g_{j}}\}$ and initial relaxation order $k=d$.
\item\label{item::2} Solve primal and dual SDP problems (\ref{eq::PR}) and
% Li change: Sedumi to SeDuMi
% Li change march 22
(\ref{eq::DPR}) by standard SDP solver (e.g., SeDuMi \cite{sedumi},
SDPT3 \cite{sdpt3}, SDPNAL \cite{sdpnal}).
% Li change: add ","
\item For $t\in [d, k]$, check condition (\ref{eq::FTC}).
\begin{enumerate}[\upshape 1]
\item If (\ref{eq::FTC}) holds for some $t$, get minimizers by Extraction
Algorithm \cite{HL05} and stop;
% Li change march 4 delete Step
\item Otherwise, go to Step \ref{item::4}.
\end{enumerate}
\item\label{item::4} If $k>k_{\max}$, stop; otherwise, set $k=k+1$ and go to
Step \ref{item::2}.
\end{enumerate}
\end{alg}

\subsection{Jacobian SDP relaxation}
The convergence of Lasserre's SDP relaxations (\ref{eq::PR}) and
(\ref{eq::DPR}) might be {\itshape asymptotic} for some instances,
i.e., only lower bounds are found for each order $k$. To overcome
this hurdle, Nie \cite{Nie2011Jacobian} proposed a refined
reformulation of (\ref{pro::opti}) by some ``Jacobian-type''
technique whose SDP relaxation has finite convergence.

Roughly speaking, Jacobian SDP relaxation is to add auxiliary
constraints to (\ref{pro::opti}) which represent optimality
conditions under the assumption that the optimum $f_{\min}$ is
achievable. The basic idea is that at each optimizer, the Jacobian
matrix of the objective function, the equality constraints and the
active inequality constraints must be singular, i.e., all its
maximal minors vanish. For convenience, denote
\[
h:=(h_1,\ldots,h_{m_1})\quad\text{and}\quad g:=(g_1,\ldots,g_{m_2}).
\]
For a subset $J=\{j_1,\ldots,j_k\}\subseteq [m_2]$, denote
$g_J:=(g_{j_1},\ldots,g_{j_k})$. Symbols $\nabla h $ and $\nabla g_J
$ represent the gradient vectors of the polynomials in $h $ and $g_J
$, respectively. Denote the determinantal variety of $(f,h,g_J)$'s
Jacobian being singular by
\[
G_J=\left\{x\in\CC^n\mid\text{rank}\ B^J(x)\le m_1+|J|\right\},\]
where
\[B^J(x)=\left[\nabla f(x)\quad\nabla h(x)\quad\nabla
g_J(x)\right].
\]
Instead of using all maximal minors to define $G_J$, \cite[Section
2.1]{Nie2011Jacobian} discusses how to get the smallest number of
defining equations. Let $\eta^J_1,\ldots,\eta^J_{len(J)}$ be the set
of defining polynomials for $G_J$ where $len(J)$ is the number of
these polynomials. For each $i=1,\ldots,len(J)$, define
\begin{equation}\label{def::varphi}
\varphi_i^J(x)=\eta^J_i\cdot\underset{j\in J^c}{\prod}g_j(x),\
\text{where}\ J^c=[m_2]\backslash J.
\end{equation}
For simplicity, we list all possible $\varphi^J_i$ in
(\ref{def::varphi}) sequentially as
% Li change: delete the formula number
\begin{equation*}\label{list::varphi}
\varphi_1,\varphi_2,\ldots,\varphi_r,\ \text{where}\
r=\underset{J\subseteq[m_2],|J|\le m-m_1}{\sum}len(J).
\end{equation*}
Consider the following optimization by adding all $\varphi_l$'s to
(\ref{pro::opti}):
\begin{equation}\label{pro::jacobian0}
\left\{
\begin{aligned}
s^*:=\underset{x\in\RR^n}{\min}&\ f(x)\\
\text{s.t.}&\ h_i(x)=0,\ i\in[m_1],\ \varphi_l(x)=0,\ l\in[r],\\
&\ g_j(x)\ge 0,\ j\in[m_2].
\end{aligned}\right.
\end{equation}
As shown in \cite[Lemma 3.1]{Nie2011Jacobian} and \cite[Lemma
3.5]{FENGUO}, by adding auxiliary constraints $\varphi_l(x)=0$, the
feasible set of (\ref{pro::jacobian0}) is restricted to the KKT
points and singular points of the feasible set of (\ref{pro::opti}).
Therefore, (\ref{pro::opti}) and (\ref{pro::jacobian0}) are
equivalent if the minimum $f_{\min}$ of {\upshape(\ref{pro::opti})}
is achievable.
\begin{lemma}\emph{(}\cite[Lemma
3.6]{FENGUO}\emph{)}\label{lemma:exact:Jacobian}
Assume $m_1\le n$ and at most $n-m_1$ of $g_1(x)$,$\ldots$,$g_{m_2}(x)$
vanish for any feasible point $x$. If the minimum $f_{\min}$ of
{\upshape(\ref{pro::opti})} is achievable, then $s^*=f_{\min}$.
\end{lemma}
\begin{remark}\label{rk::jacobian}
Since $s^*$ is the minimal value of $f(x)$
achieved among the KKT points and singular points of the feasible
set of (\ref{pro::opti}), it is possible that $s^*>f_{\min}$ (cf. \cite[Section
2.2]{Nie2011Jacobian}) if $f_{\min}$ is not achievable.
\end{remark}
If the Archimedean Condition holds for the feasible set $K$, then
$K$ is compact and $f_{\min}$ is achievable. By Lemma
\ref{lemma:exact:Jacobian}, we always have $s^*=f_{\min}$.
%We apply
Applying Lasserre's SDP relaxations (\ref{eq::PR}) and
(\ref{eq::DPR}) to (\ref{pro::jacobian0}), the resulting SDP
relaxations for (\ref{pro::jacobian0}) have finite convergence under
some {\itshape generic} conditions (cf. \cite[Theorem
4.2]{Nie2011Jacobian}, \cite[Theorem 3.9]{FENGUO}).

\begin{alg}\label{alg:finit:cons:Jacobian}Nie's Jacobian SDP relaxation

\noindent\textbf{Input:} Objective function $f(x)$, constraints
functions $h_i(x), g_j(x)$, maximal relaxation order $k_{\max}$.

\noindent\textbf{Output:} Global minimum and minimizers of problem
(\ref{pro::opti}).

\begin{enumerate}[\upshape I]
\item Construct the auxiliary polynomials $\varphi_l(x)$'s.
\item Set $d:=\max\{1, d_f, d_{h_i}, d_{g_{j}}, d_{\varphi_l}\}$ and initial relaxation
order $k=d$.
\item\label{item::3} Solve (\ref{pro::jacobian0}) by Algorithm
\ref{alg:finit:cons:putinar}.
\item For $t\in[d, k]$, check condition (\ref{eq::FTC}).
\begin{enumerate}[\upshape 1]
\item If (\ref{eq::FTC}) holds for some $t$, get minimizers by Extraction
Algorithm \cite{HL05} and stop;
% March 5 Li delete Step
\item Otherwise, go to Step \ref{item::5}.
\end{enumerate}
\item\label{item::5} If $k>k_{\max}$, stop; otherwise, set $k=k+1$ and go to
Step \ref{item::3}.
\end{enumerate}
\end{alg}

In contrast to Lasserre's SDP relaxation, Jacobian SDP relaxation is
more complicated due to the auxiliary polynomials $\varphi_l(x)$'s.
We refer to \cite[Section 4]{Nie2011Jacobian} for some simplified
versions of Jacobian SDP relaxation method.

\section{ SIPP with compact set $U$}\label{sec::compact}
The two SDP relaxation algorithms shown in Section \ref{sec::SDP}
provide strong tools to globally solve polynomial optimization
problems with finitely many constraints. In this section, we will
discuss how to use them to solve SIPP problems globally.

\subsection{A semidefinite relaxation algorithm}\label{subsec::exchange}

One main difficulty in solving a SIP problem is that there are
infinite number of constraints. How to deal with the infinite index
set $U$ is the key difference among various SIP algorithms.
Exchange method is commonly used in SIP computation, and is regarded
as the most efficient method on solving SIP problems
\cite{Hettich1993,Still2007}. The general steps of exchange method
are determined algorithmically as follows \cite{Hettich1993}. Given
a subset $U_k\subseteq U$ in iteration $k$ with $|U_k|<\infty$,
compute at least one global solution $x^k$ of
\begin{equation}\label{eq::f}
\min\limits_{x\in X}\ f(x)\quad \text{s.t.}\ g(x,u)\geq 0,
~\forall~u\in U_k,
\end{equation}
and solutions $u_1,\ldots,u_t$ of the subproblem
\begin{equation}\label{eq::g}
g^k:=\min\limits_{u\in U}\ g(x^k,u).
\end{equation}
If $g^k\ge 0$, stop; otherwise, set
$U_{k+1}=U_k\cup\{u_1,\ldots,u_t\}$ and go to next iteration.
Therefore, to successfully apply exchange method to solve SIPP
problems, we need to globally solve subproblems
(\ref{eq::f})-(\ref{eq::g}) and extract global minimizers in each
iteration. As we have discussed in Section \ref{sec::SDP}, the SDP
relaxation methods are proper means for this propose.  The specific
description of exchange method with SDP
% Li change add (P)
relaxations for SIPP problems is shown in the following.

\begin{alg}\label{alg:exchange:method} Semidefinite relaxations for
SIPP

\noindent\textbf{Input:} Objective function $f(x)$, constraint
function $g(x,u)$, semi-algebraic sets $X$, $U$, tolerance
$\epsilon$ and maximum iteration number $k_{\max}$.

\noindent\textbf{Output:} Global optimum $f^*$ and set $X^*$ of minimizers
of problem $(P)$.
\begin{enumerate}[\itshape Step 1]
\item
Choose random $u_0\in U$ and let $U_0=\{u_0\}$. Set $X^*=\emptyset$
and $k=0$.
\item\label{item::step2} Use Algorithm \ref{alg:finit:cons:putinar} to solve
\begin{equation}\label{semiinfinite:pp:sub1}
(P_k): \left\{
\begin{aligned}
 f^{\min}_k :=\min\limits_{x\in X}&\ f(x)\\
\text{s.t.}&\ g(x,u)\geq 0, ~\forall~u\in U_k.\\
\end{aligned}
\right.
\end{equation}
Let $S_k =
\{x^k_1,\cdots,x^k_{r_k}\}$ be the set of the global minimizers of problem
$(P_k)$.
\item
Set $U_{k+1}=U_k$. For $i=1,\cdots,r_k$,
\begin{enumerate}[\upshape (a)] \label{item::step3}
\item Use Algorithm \ref{alg:finit:cons:Jacobian} to solve
\begin{equation} \label{semiinfinite:pp:sub2}
(Q^k_i):\quad \left.
\begin{aligned}
 g^k_i :=\min\limits_{u\in U}&\ g(x^k_i,u).\\ % Li change add .May 9
\end{aligned}
\right.
\end{equation}
Let $T^k_i=\left\{u^k_{i,j},\ j=1,\cdots,t^k_i\right\}$ be the set
of global minimizers of $(Q^k_i)$.
\item Update $U_{k+1}=U_{k+1}\bigcup T^k_i$.
\item If $g^k_i\geq -\epsilon$, then update $X^*=X^*\bigcup \{x^k_i\}$.
\end{enumerate}
\item
If $X^*\neq \emptyset$ or $k>k_{\max}$, stop;

\noindent otherwise, set $k=k+1$ and go back to Step \ref{item::step2}.
\end{enumerate}
\end{alg}

\begin{remark}
Subproblems $(P_k)$ and $(Q_i^k)$ in Algorithm
\ref{alg:exchange:method} can be solved by both Algorithm
\ref{alg:finit:cons:putinar} and \ref{alg:finit:cons:Jacobian}.
Finite convergence can be guaranteed by Algorithm
\ref{alg:finit:cons:Jacobian} which, however, produces SDPs of size
exponentially depending on the number of the constraints. Since
$U_k$ enlarges as $k$ increases, subproblem $(P_k)$ consequently
becomes hard to be solved by Algorithm
\ref{alg:finit:cons:Jacobian}. Therefore, we solve $(P_k)$ by
Algorithm \ref{alg:finit:cons:putinar} which is also proved to have
finite convergence generically \cite{NiefiniteLas}. Because the
index set $U$ is fixed and compact, Algorithm
\ref{alg:finit:cons:Jacobian} is a better choice for solving
$(Q_i^k)$.
\end{remark}

\begin{prop}[Monotonic Property]\label{pro:monotonic}
For optimal values of $(P_k)$ in $(\ref{semiinfinite:pp:sub1})$, we
have
\begin{equation}\label{eq::monotonic}
f^{\min}_1\le\cdots\le f^{\min}_{k}\le f^{\min}_{k+1}\le\cdots\le f^*.
\end{equation}
\end{prop}
\begin{proof}
Because
\[
U_1\subseteq\cdots\subseteq U_k\subseteq
U_{k+1}\subseteq\cdots\subseteq U.
\]
 So the feasible sets of $(P_{k})$
and $(P)$ satisfy
\[
K\subseteq\cdots\subseteq K_{k+1}\subseteq
K_k\subseteq\cdots\subseteq K_1,
\]
we
obtain the conclusion.
\end{proof}

We have the following convergence analysis of Algorithm
\ref{alg:exchange:method}:
\begin{theorem}\label{theorem:exchange}

Suppose that $X$ is compact. If at each step $k$,
\begin{enumerate}[\upshape (a)]
\item\label{prop::1} subproblems $(P_k)$ and each $(Q_i^k)$ are globally
solved,
\item\label{prop::2} intermediate results $S_k$ and at least one $T_i^k$ are
nonempty,
\end{enumerate}
then either Algorithm \ref{alg:exchange:method} stops with solutions
to $(P)$ in a finite number of iterations or for any sequence
$\{x^k\} $ with $x^k\in S_k$, there exists at least one limit point
as $k$ increases and each of them solves $(P)$.
\end{theorem}
\begin{proof}
At each step, if (\ref{prop::1}) holds, then global optima $f_k^{\min}$ and
$g_i^k$ are obtained and monotonic property (\ref{eq::monotonic}) is true.
Additionally, if (\ref{prop::2}) is satisfied, then Algorithm
\ref{alg:exchange:method} either stops in a finite number of iterations or
proceeds without interrupt as $k$ increases.

If Algorithm \ref{alg:exchange:method} stops at $k$-th iteration with $k<k_{\max}$,
then $g_i^k\ge 0$ for some $i$, which implies that the
associated $x^k_i$ is feasible for $(P)$.
Moreover, $x^k_i$ is a global minimizer of $(P)$ by (\ref{eq::monotonic}).
Now we assume $g_i^k<0$ for each $k$ and $i$ which implies $T_i^k\not\subseteq
U_k$ and $U_k\subset U_{k+1}$ for all $k$. The following argument is based on
the proof of \cite[Theorem 7.2]{Hettich1993}. For any $x\in X$, define
\[
v(x):=\min\{g(x,u), u\in U\}.
\]
Obviously, $v(x)$ is continuous. Fix a sequence $\{x^k\} $ with
$x^k\in S_k$, then a limit point $\bar{x}\in X$ always exists since $X$ is
compact. Without loss of generality, assume $x^k\rightarrow \bar{x}$.
By (\ref{eq::monotonic}), it suffices to prove that $\bar{x}$ is
feasible for $(P)$. Let $v(x^k)=g(x^k,u^k)$ and $X_k$ be the feasible set of
$(P_k)$. Since $U_k\subset U_{k+1}$, we have
$\bar{x}\in\cap_{k=1}^{\infty}X_k$ and therefore $g(\bar{x},u^k)\ge 0$. Then
\[
\begin{aligned}
v(\bar{x})&=v(x^k)+[v(\bar{x})-v(x^k)]\\
&=g(x^k,u^k)+[v(\bar{x})-v(x^k)]\\
&\ge [g(x^k,u^k)-g(\bar{x},u^k)]+[v(\bar{x})-v(x^k)].
\end{aligned}
\]
By the continuity of $v$ and $g$, we have $v(\bar{x})\ge 0$, i.e.,
$\bar{x}$ is feasible for $(P)$.
\end{proof}

If $X$ and $U$ are compact, then the optima of $(P_k)$ and $(Q_i^k)$
are achievable. By applying SDP relaxations Algorithm
\ref{alg:finit:cons:putinar} and Algorithm
\ref{alg:finit:cons:Jacobian} to $(P_k)$ and $(Q_i^k)$, as we have
mentioned in Section \ref{sec::SDP}, (\ref{prop::1}) and
(\ref{prop::2}) are {\itshape generically} satisfied no matter what
initial $U_0$ we choose. In section \ref{sec::noncompact}, we will
consider the case when $U$ is noncompact for which the convergence
of Algorithm \ref{alg:exchange:method} might fail if we choose an
arbitrary initial $U_0$ (Example \ref{ex::count}). We will deal with
this issue by the technique of homogenization.

\subsection{Numerical experiments}
This subsection presents some numerical examples to illustrate the
efficiency of Algorithm \ref{alg:exchange:method}. The computation
is implemented with Matlab 7.12 on a Dell 64-bit Linux Desktop
running CentOS (5.6) with 8GB memory and Intel(R) Core(TM) i7 CPU
860 2.8GHz. Algorithm \ref{alg:exchange:method} is implemented with
software Gloptipoly \cite{Gloptipoly}. SeDuMi \cite{sedumi} is used
as a standard SDP solver. Throughout the computational experiments,
we set parameters $k_{\max}=  15$, $\epsilon = 10^{-4}$ in Algorithm
\ref{alg:exchange:method}. After Algorithm \ref{alg:exchange:method}
terminates, let $X^*$ be the output set of global minimizers of
$(P)$, $f^*$ be the value of the objective function $f$ over $X^*$
and Iter be the number of iterations Algorithm
\ref{alg:exchange:method} has proceeded. Let
\[
\mathbf{\text{Obj}_2}:=\min\limits_{x^*\in X^{*}} \min\limits_{u\in
U}~g(x^*,u).
\]
By the discussion in Subsection \ref{subsec::exchange}, the global
minimizers in $X^*$ can be certified by inequality
$\mathbf{\text{Obj}_2}\ge -\epsilon$.

\subsubsection{Examples of small SIPP problems}
We test some small examples taken from \cite[Appendix A]{Green2005}.
For nonpolynomial functions, e.g., sine, cosine or exponential
function, we use their Taylor polynomial approximations, see
Appendix \ref{appendix:small:exm}. Let $X = [-100,100]^n$. Test
results are reported in Table \ref{SIPP:small}. The Iter column in
Table \ref{SIPP:small} indicates that Algorithm
\ref{alg:exchange:method} takes a very few steps to find the global
minimizer which are certified by the $\text{Obj}_2$ column.
\begin{table}[htb]\caption{Computational results for small SIPP problems. } \label{SIPP:small}
\centering
\begin{small}
\begin{tabular}{|c|c|c|c|c|c|c|c|c|} \hline
 No. &  $x^*$  & $\text{Iter}$  & $f^*$ &  $\text{Obj}_2$   \\
\hline
Example \ref{small:example:1}  &  (-0.0008,~0.4999) &  2  & -0.2504 & 6.4744e-7\\
\hline
Example \ref{small:example:2}  &  (-0.7500,~-0.6180) &  3  & 0.1945 & 3.5305e-7 \\
\hline
Example \ref{small:example:3}   & (-0.1514,~-1.7484,~2.5725)& 2 & 9.6973 & 7.8870e-5\\
\hline
Example \ref{small:example:4} & (-1,0,0) & 2 & 1 & 6.2320e-5\\
\hline
Example \ref{small:example:5} & (0,0) & 2&0 & -1.1578e-12\\
\hline
Example \ref{small:example:6} & (0,0,0)  & 2& 4 & -4.7070e-12\\
\hline
Example \ref{small:example:7} & (0,0)  & 2& 0  & 1.9285e-12\\
\hline
\end{tabular}
\end{small}
\end{table}

\subsubsection{Examples of random SIPP problems}

We test the performance of Algorithm \ref{alg:exchange:method} on
some random SIPP problems which are generated as follows.

Let $x=(x_1,\ldots,x_n)$ and $u=(u_1,\ldots,u_p)$. Given
$d\in\mathbb{N}$, let $[x]_d$ and $[u]_d$ be the vectors of
monomials with degree up to $d$ in $\RR[x]$ and $\RR[u]$,
respectively. Denote $\langle [x]_d,[u]_d\rangle$ as the vector
obtained by stacking $[x]_d$ and $[u]_d$. Let
$f(x)=\eta^{T}[x]_{2d_1}$ be the objective function where $\eta$ is
a Gaussian random vector of matching dimension. Let $g(x,u)= \tau -
\langle [x]_{d_2},[u]_{d_2}\rangle^{T}M\langle
[x]_{d_2},[u]_{d_2}\rangle$, where $\tau$ is a random number in
$[1,10]$ and $M$ is a random positive semidefinite matrix of
matching dimension.  Let $X=B_n(0,1)$ be the unit ball in $\RR^n$
and $U$ varies among $U_1 = B_p(0,1),~U_2 = [-1,1]^p$ and $U_3 =
\Delta_p$ where $\Delta_p$ is the $p$ dimensional simplex.

The results using Algorithm \ref{alg:exchange:method} are shown in
Table \ref{T:random:SIPP:X:exchange} where the Inst column denotes
the number of randomly generated instances, the consumed computer
time is in the format hr:mn:sc with hr (resp. mn, sc) standing for
the consumed hours (resp. minutes, seconds). The column
$\text{Obj}_2$ shows that Algorithm \ref{alg:exchange:method}
successfully solves all the random problems.

\begin{table}\caption{Computational results for random SIPP problems} \label{T:random:SIPP:X:exchange}
\centering
\begin{scriptsize}
\begin{tabular}{|c|c|c|c|c|c|c|c|c|c|c|c|c|} \hline
No. & $n$ & $p$ & $d_1$& $d_2$&  Inst  &  $U$ & \multicolumn{2}{c|}{time (min, max)} &\multicolumn{2}{c|}{$\text{Obj}_2$ (min, max)}   \\
%\hline 1&   4 &3 &2 &2 & 10 & 2 & $Y_1$ &   0:00:04 & 0:00:17 & 4.3945e-009 & 1.7724  \\   %
\hline 1 &  5 & 3& 3 & 2 &10 &  $ U_1$ &   0:00:17 & 0:00:28 & 1.3479 & 2.0779  \\   %
\hline 2 &  5& 3 & 2 & 2 & 10  &  $U_3$ &   0:00:06 & 0:00:12 & -9.5236e-9  & 0.6343  \\%
\hline 3 &  6 & 2& 2 & 2 & 10 &   $U_1$ &   0:00:19 & 0:00:22 & 1.7144 &  2.1185   \\ %
\hline 4 &  6 & 3& 2 & 2& 10 &   $U_1$ &   0:00:19 & 0:00:24 & 1.0450 & 1.7220  \\ %
%\hline 6 &  6 & 2& 3 & 2& 10  & 2& $Y_1$ &   0:00:21 & 0:00:25 &  1.5508 &  2.1974   \\ %
\hline 5 &  7 & 3& 3 & 2 & 10   &   $U_3$ & 0:00:26 & 0:00:59 &
3.7797e-8  &  0.3198  \\
\hline 6 &  8 & 3& 2 & 2 &10 &   $U_1$ &   0:04:52 & 0:05:18 & 1.3213 & 1.8438   \\  %
\hline 7 &  9 & 2& 2 & 2& 5 &   $U_1$ &   0:45:26 & 0:49:28 & 1.5850 & 2.2807\\ %
\hline 8 & 9& 2& 2 & 2& 5 &  $U_3$ &   0:44:40 & 0:52:49 &
1.7521e-8&2.9119e-7   \\
\hline 9 &  5 & 2& 2 & 2 & 5 &   $U_2$ &   0:57:17 & 1:04:02 &
1.3116e-6 &  1.6986e-5   \\  \hline
\end{tabular}
\end{scriptsize}
\end{table}

\subsection{Application to PMI problems}
In this subsection, we apply Algorithm \ref{alg:exchange:method} to
the following optimization problem with {\itshape polynomial matrix
inequality} (PMI):
\begin{equation}\label{PMI:1}
f^{\min}:=\min\limits_{x\in \mR^n}\ f(x)\qquad \text{s.t.}\ \
G(x)\succeq 0,
\end{equation}
where $f(x)\in \mR[x]$ and $G(x)$ is an $m\times m$ symmetric matrix
with entries $G_{ij}(x)\in \mR[x]$. PMI is a special SIPP problem
and has been widely arising in control system design, e.g., static
output feedback design problems \cite{PMI2006}. PMI is also
interesting in optimization theory, e.g., SDP representation of a
convex semialgebra set \cite{NieLMI}. Some traditional methods for
globally solving (\ref{PMI:1}) are based on branch-and-bound schemes
and alike \cite{GOBMIP} which, as pointed in \cite{PMI2006}, are
computationally expensive.  Recently, some global methods based on
SOS relaxations are proposed in \cite{HS,MK} as well as in
\cite{GOBMIP} in a dual view.

Define
\begin{equation*}
X :=\{x\in \mR^n\mid G(x)\succeq 0\}\quad\text{and}\quad U:=\{u\in \mR^m\mid
\|u\|^2=1\}.
\end{equation*}
Then problem (\ref{PMI:1}) is equivalent to the following SIPP
problem
\begin{equation}\label{PMI:2}
\left\{
\begin{aligned}
 \min\limits_{x\in \mR^n}&\ f(x)\\
\text{s.t.}&\ g(x,u) =u^{T}G(x)u\geq0,~~\forall~u\in U.\\
\end{aligned}
\right.
\end{equation}
Assume the feasible set $X$ is compact, then we can apply Algorithm
\ref{alg:exchange:method} to solve SIPP problem (\ref{PMI:2}). The
following examples show that Algorithm \ref{alg:exchange:method} is
efficient to solve PMI problems.
\begin{exm}\label{dimension:3:PMI}
Consider the following PMI problem:
\begin{equation}\label{eq::PMIex1}
\left\{
\begin{aligned}
\min\limits_{x\in\mR^2}&\ f(x) =x_1+x_2\\
\text{s.t.}&\ G(x) = {\left[
\begin{array}{ccc}
4-x_1^2-x_2^2 & x_1 & x_2   \\
x_1 & x_2^2-x_1 & x_1x_2  \\
x_2 & x_1x_2 & x_1^2-x_2
\end{array}
\right]}\succeq 0.
\end{aligned}\right.
\end{equation}
The characteristic polynomial of matrix $G(x)$ is:
\begin{equation*}
\begin{aligned}
p(t,x) =\det(tI_{3}-G(x))=t^3-g_1(x)t^2+g_2(x)t-g_3(x)
\end{aligned}
\end{equation*}
where
\begin{equation*}
\begin{aligned}
g_1(x) & =  4-x_1-x_2, \\
g_2(x) &=  x_1^2
x_2-4x_2-x_1^4+x_1x_2-x_2^4-2x_1^2x_2^2+x_1x_2^2-4x_1+3x_1^2+3x_2^2, \\
g_3(x) & =
x_1^2x_2+4x_1x_2+2x_1^2x_2^2+x_1x_2^2-x_1^3x_2-4x_1^3+x_2^2x_1^3-x_2^3x_1-4x_2^3\\
&\quad -x_1^4-x_2^4+x_1^5+x_2^5+x_1^2x_2^3.
\end{aligned}
\end{equation*}
According to Descartes' rule of signs \cite{PMI2006}, the feasible
set of (\ref{eq::PMIex1}) is
\[
\left\{x\in\RR^2\mid g_1(x)\geq 0,~~g_{2}(x)\geq 0,~~g_3(x)\geq
0\right\}
\]
which is shown shaded in Figure \ref{PMI:example:dim3}. We first
reformulate (\ref{eq::PMIex1}) as a SIPP problem (\ref{PMI:2}), then
apply Algorithm \ref{alg:exchange:method} to it. After $5$
iterations, we get a global minimizer $x^*\approx ( -1.2853,
-1.2763)$ which is certified by $\text{Obj}_2=-1.4523\times
10^{-4}$. The accuracy of this result can be seen from Figure
\ref{PMI:example:dim3}.\hfill$\square$
\begin{figure}
\includegraphics[width=0.4\textwidth]{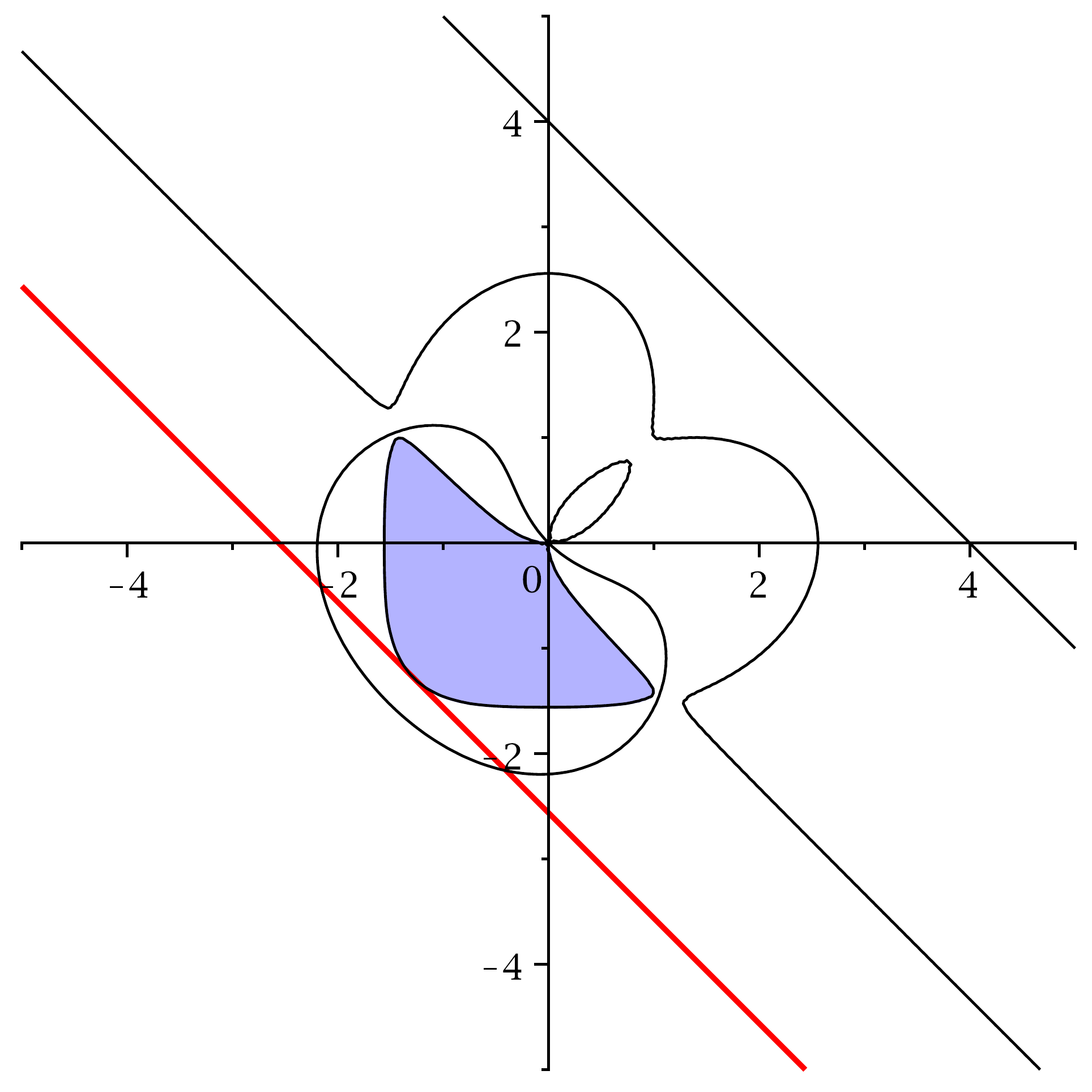}
\caption{Feasible region of PMI problem (\ref{eq::PMIex1}) in
Example \ref{dimension:3:PMI}.}\label{PMI:example:dim3}
\end{figure}
\end{exm}

\begin{exm}\label{dimension:4:PMI}
Consider the following PMI problem:
\begin{equation}\label{eq::PMIex2}
\left\{
\begin{aligned}
\min\limits_{x\in\mR^2}&\ f(x) = x_1-x_2 \\
\text{s.t.}&\ G(x) = {\left[
\begin{array}{cccc}
10-x_1^2-x_2^2 & x_1  & -x_1^2+x_2 & x_2+3   \\
x_1  & x^2_2  & x_1-x^2_2 & x_1  \\
-x^2_1+x_2 & x_1-x_2^2 & x_1+2x^2_2& x_2 \\
 x_2+3 & x_1 & x_2 & x^2_2
\end{array}
\right]}\succeq 0.
\end{aligned}\right.
\end{equation}

Similar to Example \ref{dimension:3:PMI}, we obtain the feasible set
of (\ref{eq::PMIex2}) by Descartes' rule of signs \cite{PMI2006} and
show it shaded in Figure \ref{PMI:example:dim4}. Applying Algorithm
\ref{alg:exchange:method} to the reformulation (\ref{PMI:2}) of
problem (\ref{eq::PMIex2}), we get global minimizer
$x^*\approx(0.5093,-1.0678)$ and minimum $f(x^*)\approx 1.5771$
which are certified by $\text{Obj}_2=-9.4692\times 10^{-5}$. From
Figure \ref{PMI:example:dim4}, we can see this result is
accurate.\hfill$\square$
\begin{figure}
\includegraphics[width=0.4\textwidth]{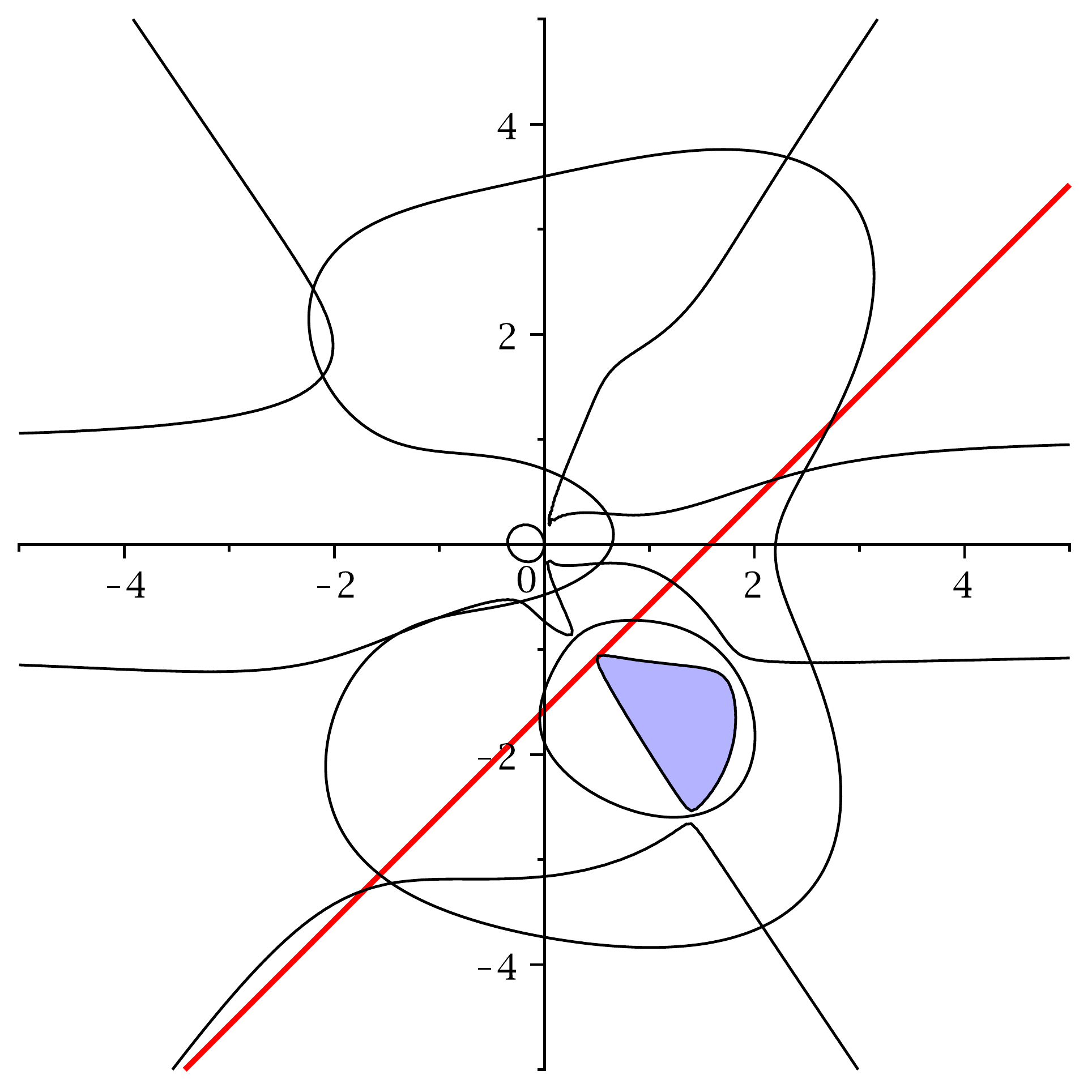}
\caption{Feasible region of PMI
%%%constraint
problem (\ref{eq::PMIex2}) in Example
\ref{dimension:4:PMI}.}\label{PMI:example:dim4}
\end{figure}
\end{exm}
We end this subsection by pointing out a trick hidden in the
reformulation (\ref{PMI:2}) of (\ref{PMI:1}). PMI optimization
problem (\ref{PMI:1}) can be regarded as a SIPP problem with
noncompact index set $\widetilde{U}=\RR^m$. Since the constraint
function $g(x,u)$ is homogenous in $u$, we can restrict
$\widetilde{U}$ to the unit sphere $U$. By Theorem
\ref{theorem:exchange}, to guarantee the convergence of Algorithm
\ref{alg:exchange:method}, the optimum of $(Q_i^k)$ needs to be
achievable for each $k$ which might fail if $U$ is noncompact. The
reformulation (\ref{PMI:2}) of (\ref{PMI:1}) gives us a clue for
dealing with SIPP with noncompact $U$ by the technique of
homogenization. We will go into detail about this technique in next
section.

\section{SIPP with noncompact set $U$ }\label{sec::noncompact}
At some $k$-th iteration of Algorithm \ref{alg:exchange:method}, if
the global minima $g_i^k$ of $(Q_i^k)$ are not achievable for all
$x_i^k\in S_k$, then by Remark \ref{rk::jacobian}, either
\begin{enumerate}[{case} 1.]
\item $T_i^k=\emptyset$, then $U_{k+1}$ can not be updated and
consequently $S_{k+1}$ remains the same as $S_k$, or
\item $U_{k+1}$
is updated by $T_i^k$ which consists of KKT points or singular
points of the feasible set of $(Q_i^k)$ rather than global
minimizers.
\end{enumerate}
As we have discussed in Subsection \ref{subsec::exchange}, the
convergence property of Algorithm \ref{alg:exchange:method} might
fail or wrong global minimizers might be outputted if the above
cases happen. For example,
\begin{example}\label{ex::count} Consider the following problem:
\begin{equation}\label{eq::ce}
\left\{
\begin{aligned}
f^*:=\underset{x_1,x_2\in\mR}{\min}&\ -x_1-x_2\\
\text{s.t.}&\ x_1(u_1^2-1)+(x_2-u_1u_2)^2\ge 0,\ \forall~u_1,u_2\in\mR,\\
&\ x_1^2+x_2^2=2.
\end{aligned}\right.
\end{equation}
We choose $u_1, u_2$ such that $x_2-u_1u_2=0$. By letting $u_1$ tend
to infinity and $0$ respectively, we obtain that $x_1=0$ for any
feasible point $x$. Therefore, there are only two feasible points
$(0, \pm\sqrt{2})$ and the global minimum is $-\sqrt{2}$ with
minimizer $(0,\sqrt{2})$.

We claim that Algorithm \ref{alg:exchange:method} {\itshape fails}
to solve (\ref{eq::ce}) if we set initial
$U_0=\{(u^{0}_1,u^{0}_2)\}$ such that
\[
(u^0_1, u^0_2)\notin \mathcal{U}:=\{u\in\RR^2\mid u_1u_2=\sqrt{2},
u_1^2< 2\sqrt{2}-2\}.
\]
We prove it in the following. First, we show that for any
$(u_1,u_2)\in\RR^2$ there always exists $(\bar{x}_1, \bar{x}_2)$
with $\bar{x}_1>0, \bar{x}_2>0$ such that
\[
g(\bar{x}, u):=\bar{x}_1(u_1^2-1)+(\bar{x}_2-u_1u_2)^2\ge 0, \quad
\bar{x}_1^2+\bar{x}_2^2=2.
\]
This is true if $g((0,\sqrt{2}),u)>0$ or $g((\sqrt{2},0),u)>0$ by
the continuity of $g(x,u)$. Now we assume
\[
g((0,\sqrt{2}),u)\le 0\quad\text{and}\quad g((\sqrt{2},0),u)\le 0.
\]
From the first inequality, we get $u_1u_2=\sqrt{2}$. Then by the second
inequality, we have $u_1^2\le 1-\sqrt{2}$ which is a contradiction. Therefore,
the following subproblem
\begin{equation*}
(P_0): \left\{
\begin{aligned}
 f^{\min}_0 :=\min\limits_{x\in\mR^2}&\ -x_1-x_2\\
\text{s.t.}&\ x_1^2+x_2^2=2,\ g(x,u)\geq 0,\ \forall~u_1,u_2\in\mR,\\
\end{aligned}
\right.
\end{equation*}
has global minimizer $S_0=\{(\tilde{x}_1,\tilde{x}_2)\}$ with
$\tilde{x}_1>0,\tilde{x}_2>0$. Then we solve subproblem
\begin{equation}
(Q^0):\quad \left.
\begin{aligned}
 g^0 :=\min\limits_{u\in\mR^2}&\
g(\tilde{x},u)=\tilde{x}_1(u_1^2-1)+(\tilde{x}_2-u_1u_2)^2.\\
\end{aligned}
\right.
\end{equation}

Obviously, $g^0=-\tilde{x}_1$ is not achievable. Applying Jacobian
SDP relaxation Algorithm \ref{alg:finit:cons:Jacobian}, we obtain
$T^0=\{(0,0)\}$ which consists of the only critical point $(0,0)$ of
map $g(x^0,u)$ with critical value $\tilde{x}_2^2-\tilde{x}_1$. If
$\tilde{x}_2^2-\tilde{x}_1\ge 0$, then Algorithm
\ref{alg:exchange:method} terminates and outputs
$X^*=\{(\tilde{x}_1,\tilde{x}_2)\}$ which is a wrong solution. Now
we assume $\tilde{x}_2^2-\tilde{x}_1<0$ and continue. By Algorithm
\ref{alg:exchange:method}, $U_1=\{(\bar{u}_1,\bar{u}_2), (0,0)\}$.
Then we go to the next iteration and solve
\begin{equation*}
(P_1): \left\{
\begin{aligned}
 f^{\min}_1 :=\min\limits_{x\in\mR^2}&\ -x_1-x_2\\
\text{s.t.}&\ x_2^2-x_1\ge 0,\ g(x,\bar{u})\geq 0,\\
&\ x_1^2+x_2^2=2.
\end{aligned}
\right.
\end{equation*}
Let $K_1$ be the feasible set of $(P_1)$, then
\begin{enumerate}[{case} 1.]
\item There exists no $(\bar{x}_1,\bar{x}_2)\in K_1$ with
$\bar{x}_1>0,\bar{x}_2>0$. The global minimizer of $(P_1)$ is
$S_1=\{(0,\sqrt{2})\}$ and
$g^1:=\underset{u\in\RR^2}{\min}g((0,\sqrt{2}),u)\ge 0$. Therefore,
the correct global solution of (\ref{eq::ce}) is outputted. In this
case, by the continuity of $g(x,u)$, we have
$g((0,\sqrt{2}),\bar{u})\le 0$ and $g((1,1),\bar{u})<0$. From these
two inequalities, we get $(\bar{u}_1,\bar{u}_2)\in \mathcal{U}$.
\item There exists $(\bar{x}_1,\bar{x}_2)\in K_1$ with
$\bar{x}_1>0,\bar{x}_2>0$. Then the global minimizer of $(P_1)$ is
$S_1=\{(\hat{x}_1,\hat{x}_2)\}$ with $\hat{x}_1>0,\hat{x}_2>0$.
Similar to $g^0$, $g^1$ is not achievable and
$U_1=\{(\bar{u}_1,\bar{u}_2), (0,0)\}$ can not be updated.
Consequently, the same process will be repeated in the following
iterations.
\end{enumerate}
Now we have proved the claim. Since the set $\mathcal{U}$ is a
subset of a Zariski closed set of $\RR^2$, Algorithm
\ref{alg:exchange:method} fails if we choose a generic initial
$U_0=\{(u_1,u_2)\}$.\hfill$\square$
\end{example}

Hence, Algorithm \ref{alg:exchange:method} might fail to solve SIPP
problem $(P)$ if the optima of subproblems $(Q_i^k)$ can not be
reached for all $x_i^k\in S_k$ which might happen when $U$ is
noncompact. As we have mentioned at the end of Section
\ref{sec::compact}, the reformulation (\ref{PMI:2}) of (\ref{PMI:1})
sheds light on this issue by the technique of homogenization. In the
following, we apply this technique to general SIPP problem $(P)$
with noncompact index set $U$.

For given polynomial $q(u)\in\mR[u]:=\mR[u_1,\ldots,u_p]$ with
degree $d=\deg(q)$, let $\tilde{q}(\tilde{u})=u_{0}^{d}q( u/u_0 )$
be the homogenization of $q(u)$ where $\tilde{u}=(u_0,u)\in
\mR^{p+1}$. Define
\[
\tilde{g}(x,\tilde{u})=u_{0}^{d_g}g(x, u/u_0 )\ \ \text{where}\ \
d_g=\deg_u{g(x,u)}
\]
and
\begin{equation*}
\begin{aligned}
U=&\ \{u\in \mR^p| h_{1}(u)\geq 0,\cdots,h_{m_{1}}(u)\geq 0\},\\
U_{0} =&\ \{\tilde{u}\in \mR^{p+1}|\tilde{h}_{1}(\tilde{u})\geq
0,\cdots,\tilde{h}_{m_{1}}(\tilde{u})\geq 0,u_0>0,\|\tilde{u}\|^2=1\},\\
\widetilde{U} =&\ \{\tilde{u}\in
\mR^{p+1}|\tilde{h}_{1}(\tilde{u})\geq
0,\cdots,\tilde{h}_{m_{1}}(\tilde{u})\geq 0,u_0\geq
0,\|\tilde{u}\|^2=1\}.
\end{aligned}
\end{equation*}

\begin{prop}\label{prop:equivalent}
$q(u)\geq 0$ on $U$
if and only if $\tilde{q}(\tilde{u})\geq 0$ on ${\sf closure}(U_0)$.
\end{prop}
\begin{proof}

 ``If " direction.
Suppose there exists $v\in U$ such that $q(v)<0$. For $i\in [m_1]$,
we have $h_i(v)\geq 0$.  Let
 $\tilde{v}=(\frac{1}{\sqrt{1+\|v\|^2}},\frac{v}{\sqrt{1+\|v\|^2}})$,
then
\begin{equation*}\label{def:hi:hom}
\tilde{h}_{i}(\tilde{v})=(1+\|v\|^2)^{-\frac{\deg(h_i)}{2}}h_i(v)\geq
0, \quad i\in [m_1],
\end{equation*}
which implies $\tilde{v}\in U_0$ and
\begin{equation*}\label{def:G:hom}
\tilde{q}(\tilde{v}) =
 (1+\|v\|^2)^{-\frac{d}{2}}q(v)<0.
\end{equation*}
It contradicts the assumption that $\tilde{q}(\tilde{v})\geq 0$ on
${\sf closure}(U_0)$.

 ``Only if"  direction. Let $\tilde{v} = (v_0,v)\in {\sf
 closure}(U_0)$,
then there exists a sequence $\tilde{v}^k=(v^k_0,v^k)\in U_0$ such
that $\underset{k\rightarrow\infty}{\lim}(v^k_0,v^k)= (v_0,v)$ with
$v^k_0> 0$ for all $k$. We have
% Li change
\begin{equation*}
h_i(v^k/v_0^k ) = ({v_0^k})^{-\deg(h_i)} \tilde{h}_i(\tilde{v}^k)\ge
0,\ \ i\in [m_1], \ \ \text{for all}\ k.
 \end{equation*}
Therefore, the sequence $\{v^k/v^k_0\}\in U$ and $q( v^k/v^k_0 )\geq
0$. Since $q$ is continuous,
 \begin{equation*}
 \tilde{q}(\tilde{v}) = \lim\limits_{k\rightarrow \infty} \tilde{q}(\tilde{v}^k) = \lim\limits_{k\rightarrow \infty} (v^k_0)^{d}q( v^k/v^k_0 )\geq 0,
 \end{equation*}
which shows $\tilde{q}(\tilde{v})\geq 0$ on ${\sf closure}(U_0)$.
The proof is completed.
\end{proof}
\begin{cor}\label{cor::Rp}
A polynomial $q(u)\geq 0$ on $\mR^p$ if and only if
$\tilde{q}(\tilde{u})\geq 0$ on $\{\tilde{u}\in\mR^{p+1}\mid
\Vert\tilde{u}\Vert^2=1\}$.
\end{cor}
\begin{proof}
From the proof of Proposition \ref{prop:equivalent}, we can see the
inequality $u_0>0$ can be removed from $U_0$ such that $q(u)\geq 0$
on $\mR^p$ if and only if $\tilde{q}(\tilde{u})\geq 0$ on
\begin{equation*}
{\sf closure}(\{\tilde{u}\in\mR^{p+1}\mid\Vert\tilde{u}\Vert^2=1\})
=\{\tilde{u}\in\mR^{p+1}\mid\Vert\tilde{u}\Vert^2=1\}.
\end{equation*}
\end{proof}
By Proposition \ref{prop:equivalent}, we have the following
equivalent reformulation of problem $(P)$:
\begin{equation*}\label{semiinfinite:pp:noncompact:Y}
(P_0):\left\{
\begin{aligned}
 f^* := \min\limits_{x\in X}&\ f(x)\\
\text{s.t.}&\ \tilde{g}(x,\tilde{u})\geq0, ~\forall~\tilde{u}\in
{\sf closure}(U_{0}).\\
\end{aligned}
\right.
\end{equation*}
Some natural questions arise: how to get the explicit expression of
semi-algebraic set ${\sf closure }(U_0)$? Is it true that  ${\sf
closure}(U_0)=\widetilde{U}$? Clearly, we have
\begin{equation}\label{eq::YY}
{\sf closure}(U_0)\subseteq\widetilde{U}.
\end{equation}
Unfortunately, the equality does not always hold even if set $U$ is
compact (cf. \cite[Example 5.2]{NieDiscri}).
\begin{defi}\emph{(}\cite{NieDiscri}\emph{)}
$U$ is closed at
$\infty$ if ${\sf closure}(U_0)=\widetilde{U}$.
\end{defi}

Since it might be hard to express ${\sf closure}(U_0)$ for a given
particular SIPP problem, we consider to solve the following problem
in general:
\begin{equation*}\label{semiinfinite:pp:noncompact:Y:re}
(\widetilde{P}):\left\{
\begin{aligned}
 \tilde{f}^* := \min\limits_{x\in X} &\ f(x)\\
\text{s.t.}&\ \tilde{g}(x,\tilde{u})\geq0, ~\forall~\tilde{u}\in \widetilde{U}.\\
%&\ x\in X.
\end{aligned}
\right.
\end{equation*}
As set $\widetilde{U}$ is compact, the semidefinite relaxation
Algorithm \ref{alg:exchange:method} in Section \ref{sec::compact}
can successfully solve this problem with any arbitrary initial
$U_0$. Next we investigate the relation between problem $(P)$ and
problem $(\widetilde{P})$.

We define
\[
\begin{aligned}
M&=\{x\in \mR^n|g(x,u)\geq 0, \ \forall\ u\in U\}.\\
\widetilde{M}&=\{x\in \mR^n|\tilde{g}(x,\tilde{u})\geq
0,~\forall~\tilde{u}\in \widetilde{U}\}.
\end{aligned}
\]
\begin{prop}\label{pro:nonempty:set}
We have $\widetilde{M}\subseteq M$ and the equality holds if $U$ is
closed at $\infty$.
\end{prop}

\begin{proof}
By Proposition \ref{prop:equivalent}, we have
\[
M=\{x\in \mR^n|\tilde{g}(x,\tilde{u})\geq 0, \ \forall\ \tilde{u}\in
{\sf closure}(U_0)\}.
\]
Then the conclusion follows due to the relationship (\ref{eq::YY}).
\end{proof}
Consequently, we have
\begin{theorem}\label{theorem:equivalent:non}
$\tilde{f}^*\geq f^*$ and the equality holds if $U$ is closed at
$\infty$.
\end{theorem}

Corollary \ref{cor::Rp} shows that $U=\mR^p$ is closed at $\infty$
and therefore,
\begin{cor}The following two problems are equivalent:
\[
\left\{
\begin{aligned}
\min\limits_{x\in X} &\ f(x)\\
\text{s.t.}&\ g(x,u)\geq0, ~\forall~u\in\mR^p,\\
\end{aligned}
\right.\qquad
\left\{
\begin{aligned}
\min\limits_{x\in X} &\ f(x)\\
\text{s.t.}&\ \tilde{g}(x,\tilde{u})\geq0, ~\forall~\tilde{u}\in
\widetilde{U},\\
\end{aligned}
\right.
\]
where
$\widetilde{U}=\{\tilde{u}\in\mR^{p+1}\mid\Vert\tilde{u}\Vert^2=1\}$.
\end{cor}
\addtocounter{theorem}{-7}
\begin{example}[Continued]\label{counterexample}
We reformulate the problem (\ref{eq::ce}) as
\begin{equation}\label{eq::ce2}
\left\{
\begin{aligned}
\tilde{f}^*:=\underset{x_1,x_2\in\mR}{\min}&\ -x_1-x_2\\
\text{s.t.}&\ x_1(u_1^2-u_0^2)+(x_2u_0^2-u_1u_2)^2\ge 0,\
\forall~\tilde{u}\in\widetilde{U},\\
&\ x_1^2+x_2^2=2,
\end{aligned}\right.
\end{equation}
where $\widetilde{U}=\{(u_0,u_1,u_2)\in\mR^3\mid
u_0^2+u_1^2+u_2^2=1\}$. By choosing $u_0=1$, we know
$\widetilde{M}\supseteq\{(0,\pm\sqrt{2})\}$ which, obviously, are
feasible to (\ref{eq::ce2}). Therefore, $\tilde{f}^*=f^*=-\sqrt{2}$
with minimizer $(0,\sqrt{2})$. Choosing $U_0=\{(1,0,0)\}$ in
Algorithm \ref{alg:exchange:method}, Figure \ref{ex} shows the
feasible regions of subproblems $(P_k)$ for iterations
$k=0,1,\cdots,5$. Let $h(x)=x_1^2+x_2^2-2$. At $i$-th iteration, the
feasible region is defined by
\[
K_i:=\{x\in\mR^2\mid h(x)=0, g_0(x)\ge 0,\ldots,g_i(x)\ge 0\}
\]
where
\[
\begin{aligned}
g_0&=-x_1+x_2^2,\\
g_1&\approx 0.026046-0.31963x_1-0.19679x_2+0.37171x_2^2, \\
g_2&\approx 0.054893-0.11577x_1-0.18811x_2+0.16116x_2^2, \\
g_3&\approx 0.06865-0.049084x_1-0.14992x_2+0.081854x_2^2, \\
g_4&\approx 0.072498-0.025711x_1-0.12039x_2+0.049977x_2^2, \\
g_5&\approx 0.073368-0.018151x_1-0.10683x_2+0.038891x_2^2.\\
\end{aligned}
\]
For each $i$, the feasible region $K_i$ is the intersection of the
left parts of the circle $x_1^2+x_2^2=2$ devided by hyperbolas
$g_i(x)=0, i=0,\cdots,5$. From Figure \ref{ex}, we can see the
minimizers of subproblems $(P_k)$ converge to $(0,\sqrt{2})$ which
is the minimizer of problem (\ref{eq::ce}).
\begin{figure}
\includegraphics[width=0.45\textwidth]{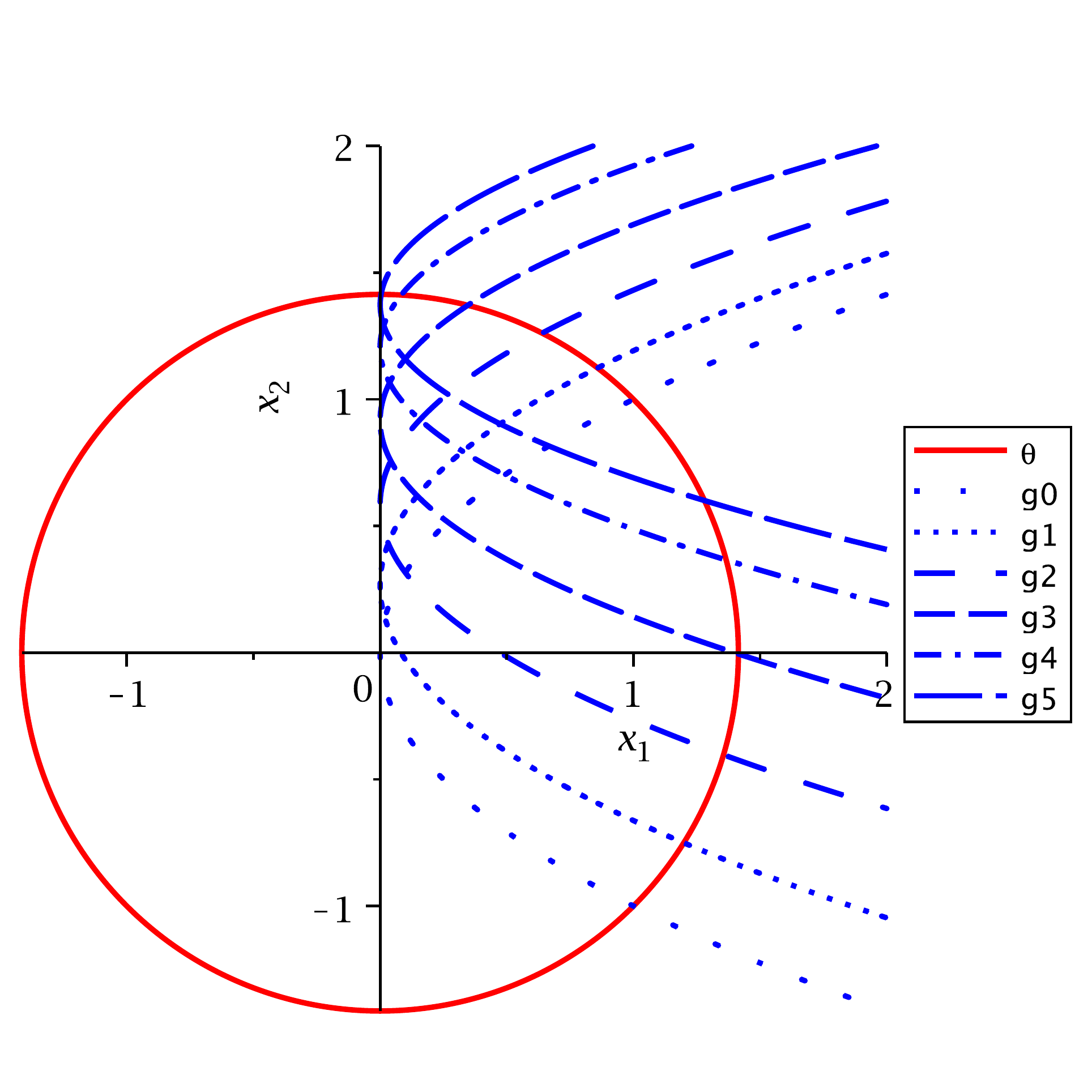}
\caption{Feasible region of Example \ref{counterexample} at each
iteration.
%in exchange method.
}\label{ex}
\end{figure}
\hfill$\square$
\end{example}
\addtocounter{theorem}{6}  We would like to point out that if $U$ is
not closed at $\infty$, we might have $\tilde{f}^*>f^*$. For
example,
\begin{example}\label{counter:example}
Consider the following SIPP problem:
\begin{equation}\label{notclose:example}
\left\{
\begin{aligned}
f^*:= \min\limits_{x\in \mR}&\ x^2\\
\text{s.t.}&\ x(u_1-u_2+1)\geq 0, ~\forall~u\in U,\\
&\ x\in [1,2],
\end{aligned}
\right.
\end{equation}
where
\[
U=\{u\in\mR^2:~u^2_1(u_1-u_2)-1=0\}.
\]
Since for all $u\in U$,
\[
g(1,u)=u_1-u_2+1 = \frac{1}{u^2_1}+1>0,
\]
$x^*=1$ is feasible and furthermore the minimizer of problem
(\ref{notclose:example}). Hence, $f^*=1$. By definition,
\begin{equation*}
\widetilde{U} = \{\tilde{u}\in
\mR^3:~u^2_1(u_1-u_2)-u^3_0=0,~u_0\geq 0,~u^2_0+u^2_1+u^2_2=1\}.
\end{equation*}
As is shown in \cite{FENGUO,NieDiscri}, $U$ is not closed at
$\infty$ because there exists a point $(0,0,1)\in \widetilde{U}$ but
$(0,0,1)\notin {\sf closure}(U_{0})$. Since for any $x\in [1,2]$,
\[
\tilde{g}(x,(0,0,1))=-x<0,
\]
we have $\widetilde{M}=\emptyset$. Therefore, $\tilde{f}^* =\infty
>f^*$.\hfill$\square$
\end{example}
Example \ref{counter:example} shows that the problem
$(\widetilde{P})$ might not be equivalent to $(P)$ when set $U$ is
not closed at $\infty$. In the following, however, we show that $U$
is closed at $\infty$ in general. In other words, $U$ is closed at
$\infty$ if it is defined by generic polynomials.

Suppose that $U$ is not closed at $\infty$, then by definition there
exists $(0,\bar{u})\in \widetilde{U}\backslash {\sf closure}(U_0)$
with $0\neq \bar{u}\in \mR^p$. Let $\hat{h}_i$ denote the
homogeneous part of highest degree of $h_i$ for $i\in[m_1]$ and
\[
\{j_1,\ldots,j_{\ell}\}:=\{j\in
[m_1]\mid\tilde{h}_j(0,\bar{u})=\hat{h}_j(\bar{u})=0\}.
\]
Then $\bar{u}$ is a solution to the polynomial system
\begin{equation}\label{eq::J}
\hat{h}_{j_1}(\bar{u})=\cdots=\hat{h}_{j_{\ell}}(\bar{u})=\Vert
\bar{u}\Vert^2-1=0.
\end{equation}
The Jacobian matrix of the system (\ref{eq::J}) at $\bar{u}$ is
\begin{equation*}
A(u):= \bbm \frac{\partial \hat{h}_{j_1}}{\partial u_1}(\bar{u}) &
\cdots & \frac{\partial\hat{h}_{j_1}}{\partial u_p}(\bar{u})
\\ \vdots & \vdots & \vdots \\
\frac{\partial \hat{h}_{j_{\ell}}}{\partial u_1}(\bar{u}) & \cdots &
\frac{\partial \hat{h}_{j_{\ell}}}{\partial u_p}(\bar{u})\\
2\bar{u}_1 & \cdots & 2\bar{u}_p
 \ebm
\end{equation*}
\begin{lemma} \emph{(}\cite[Lemma 2.10]{FENGUO}\emph{)} Suppose $U$ is not
closed at $\infty$ and
%feng changed at
%Thu Feb 21 15:43:23 PST 2013
%$|J(u)|\leq n-1$,
$\ell <p$, then $\text{rank} ~A(u)<\ell+1$.
\end{lemma}
Let $\hat{h}_{m_1+1} := \Vert\tilde{u}\Vert^2-1$ and
$J(\bar{u})=\{j_1,\ldots,j_{\ell},m_1+1\}$. We review some
background about \emph{resultants} and \emph{discriminants} in
Appendix \ref{appendix}. By Proposition \ref{prop::res4inhom} and
Proposition \ref{prop::dis4inhom}, we have
\begin{theorem} If $U$ is not closed at $\infty$, then
\begin{enumerate}[\upshape (a)]
\item If $|J(\bar{u})|>p$, then for every subset
$\{j_1,\cdots,j_{p+1}\}\subseteq J(\bar{u})$,
\[
Res(\hat{h}_{j_1},\cdots,\hat{h}_{j_{p+1}})=0.
\]
\item If $|J(\bar{u})|\le p$, then
$\Delta(\hat{h}_{j_1},\cdots,\hat{h}_{j_{\ell}},\hat{h}_{m_1+1})
=0$.
\end{enumerate}
\end{theorem}
The above theorem shows that if $U$ is defined by some generic
polynomials, then it is closed at $\infty$. Hence, the assumption
that $U$ is closed at $\infty$ is a generic condition. Therefore,
SIPP problems $(P)$ and $(\widetilde{P})$ are {\itshape equivalent}
in general.

\begin{exm}\label{example:noncompact:1} Consider the following problem
 \begin{equation}\label{eq::noncompact}
 \left\{ \begin{aligned}
\min\limits_{x\in X}&\ f(x)=x^2_1+x^2_2\\
\text{s.t.}&\ g(x,u)=x_1u_1+u_2+x_2\geq 0,\ \ \forall~u\in U,\\
\end{aligned}
\right.
\end{equation}
where
\[
X:=\{(x_1,x_2)\in \RR^2\mid x^2_1+x^2_2\leq 4\}\ \ \text{and}\ \
U:=\{(u_1,u_2)\in \RR^2\mid u^3_1+u^3_2-3u_1u_2\geq 0\}.
\]
The set $U$ is shown shaded in Figure \ref{figure:h:with:x}. Since
$u_1+u_2+1=0$ is the asymptote of the curve $u_1^3+u_2^3-3u_1u_2=0$,
the inequality $g(x,u)\ge 0$ for all $u\in U$ requires $x_1=1$ and
$x_2\ge 1$. Therefore, the feasible set of (\ref{eq::noncompact}) is
$\{x\in\RR^2\mid x_1=1, 1\le x_2\le\sqrt{3}\}$ and the global
minimizer is $x^*=(1, 1)$. It is easy to see that for a given
$(\bar{x}_1,\bar{x}_2)\in X$, the global minimum of $g(\bar{x},u)$
over $U$ is either $-\infty$ or finite but not achievable.
Therefore, by the discussion at the beginning of this section,
Algorithm \ref{alg:exchange:method} might fail to solve
(\ref{eq::noncompact}). For example, if we set $U_0=\{(1, -1)\}$,
then we get minimizer $X^*=\{(0.5000,0.4999)\}$; if $U_0=\{(1,
0)\}$, then $X^*=\{(0.0262, 0.3086)\times 10^{-5}\}$.

Now we use the homogenization technique to reformulate
(\ref{eq::noncompact}). First, we show that $U$ is closed at
$\infty$. Let
\begin{equation*}
\begin{aligned}
U_0&=\{(u_0,u_1,u_2)\in \RR^3|u^3_1+u^3_2-3u_1u_2u_0\geq
0,~u^2_0+u^2_1+u^2_2=1,~u_0>0\},\\
\widetilde{U}&=\{(u_0,u_1,u_2)\in \RR^3|u^3_1+u^3_2-3u_1u_2u_0\geq
0,~u^2_0+u^2_1+u^2_2=1,~u_0\geq 0\}.
\end{aligned}
\end{equation*}
By definition, if $U$ is not closed at $\infty$, then there exists
$(0,\bar{u}_1,\bar{u}_2)\in \widetilde{U}\backslash{\sf
closure}(U_0)$ which implies
\begin{equation*}
\bar{u}^3_1+\bar{u}^3_2=0,\quad \bar{u}^2_1+\bar{u}^2_2=1.
\end{equation*}
Therefore
\begin{equation*}
(\bar{u}_1,\bar{u}_2)\in\left\{\left(-\frac{\sqrt{2}}{2},\frac{\sqrt{2}}{2}\right),\
\left(\frac{\sqrt{2}}{2},-\frac{\sqrt{2}}{2}\right)\right\}.
\end{equation*}
Let
\[
\tilde{u}_k:=\left(\sqrt{2\varepsilon_k},-\sqrt{\frac{1}{2}-\varepsilon_k},
\sqrt{\frac{1}{2}-\varepsilon_k}\right),\quad
\hat{u}_k:=\left(\sqrt{2\varepsilon_k},\sqrt{\frac{1}{2}-\varepsilon_k},
-\sqrt{\frac{1}{2}-\varepsilon_k}\right).
\]
Let $\varepsilon_k\rightarrow 0$, then $\tilde{u}_k, \hat{u}_k\in
U_0$ for all $k$ large enough and
\[
\lim\limits_{k\rightarrow\infty}\tilde{u}_k=\left(0,-\frac{\sqrt{2}}{2},\frac{\sqrt{2}}{2}\right),
\quad
\lim\limits_{k\rightarrow\infty}\hat{u}_k=\left(0,\frac{\sqrt{2}}{2},-\frac{\sqrt{2}}{2}\right).
\]
This shows $U$ is closed at $\infty$. Therefore, by homogenization,
we reformulate (\ref{eq::noncompact}) as the following equivalent
problem
\begin{equation*}\label{one:example:noncompact:re}
 \left\{
\begin{aligned}
  \min\limits_{x\in X} &\ x^2_1+x^2_2\\
\text{s.t.}&\ \tilde{g}(x,\tilde{u})=x_1u_1+u_2+x_2u_0\geq 0,~ \tilde{u}\in \widetilde{U}.\\
%&\ x\in X.
\end{aligned}
\right.
\end{equation*}
By Algorithm \ref{alg:exchange:method}, we find a global minimizer
\[
x^*\approx (0.9999,0.9998)\quad
\text{with}\quad \text{Obj}_2=-9.8148\times 10^{-7},
\]
after several iterations. \hfill$\square$
 \begin{figure}
\includegraphics[width=0.35\textwidth]{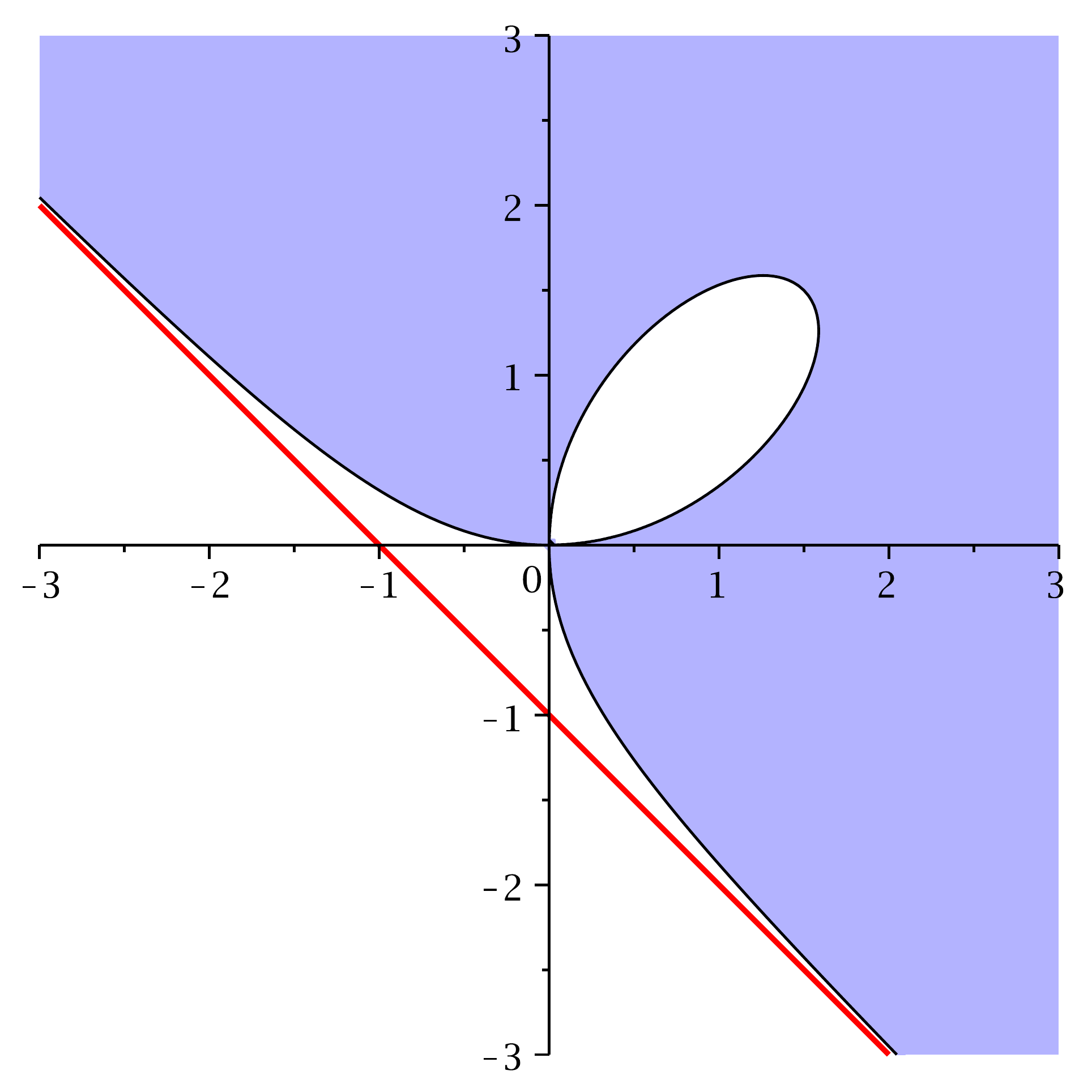}
\caption{ The feasible region $U$ in Example
\ref{example:noncompact:1}.}\label{figure:h:with:x}
\end{figure}
\end{exm}
In this section, by homogenization technique, we reformulate the SIPP problem
$(P)$ with noncompact index set $U$ as the problem $(\widetilde{P})$ with
compact index set $\widetilde{U}$ which can be globally solved by Algorithm
\ref{alg:exchange:method}.  Under the assumption that set $U$ is closed at
$\infty$ which is a generic condition, we show the two problems are equivalent.

\appendix
\section{Small SIPP examples}\label{appendix:small:exm}
\begin{exm}\label{small:example:1}
% Li change
Let $U=[0,2]$ and
 \[
f(x)=\frac{1}{3}x^2_1 + \frac{1}{2} x_1+ x^2_2-x_2,\quad
g(x,u)=-x^2_1-2x_1x_2u^2+\sin(u).
\]
%\end{aligned}
%\end{equation*}
Replace the function $\sin(u)$
%with
by $u-\frac{u^3}{6}$.
\end{exm}

\begin{exm}\label{small:example:2}
 Let $U=[0,1]$ and
 \[
f(x)=\frac{1}{3}x^2_1 + x^2_2+\frac{1}{2}x_1,\quad
g(x,u)=-(1-x^2_1u^2)^2+x_1u^2+x_2^2-x_2.
\]
  \end{exm}

 \begin{exm}\label{small:example:3}
 Let $U=[0,1]$ and
 \[
f(x)=x^2_1 + x^2_2 + x^2_3,\quad
g(x,u)=-x_1-x_2e^{x_3u}-e^{2u}+2\sin(4u).
\]
Replace function $e^{x_3u}$ by $ 1+ x_3u + \frac{1}{2}x^2_3u^2 +
\frac{1}{6}x^3_3u^3+\frac{1}{24}x^4_3u^4$, function $e^{2u}$ by $
1+ 2u+ 2u^2 +\frac{4}{3}u^3 +\frac{2}{3}u^4 $, and function
$\sin(4u)$ by $4u-\frac{32}{3}u^3$.
\end{exm}

 \begin{exm}\label{small:example:4}
Let $U=[0,1]^2$ and
\[
f(x)= x_1^2+ x^2_2 + x^2_3,\quad
g(x,u)=-x_1(u_1+u^2_2+1)-x_2(u_1u_2-u_2^2)-x_3(u_1u_2+u^2_2+u_2)-1.
\]
 \end{exm}

\begin{exm}\label{small:example:5}
Let $U=[0,\pi]$ and
\[
f(x)= x^2_2 - 4x_2,\quad g(x,u)=-x_1\cos(u)-x_2\sin(u)+1.
\]
Replace function $\sin(u)$ by $u-\frac{1}{6}u^3$ and $\cos(u)$ by
$1-\frac{1}{2}u^2 + \frac{1}{24}u^4$.
\end{exm}

\begin{exm}\label{small:example:6}
Let $U=[0,\pi]$ and
\begin{equation*}
\begin{aligned}
%Y&=[0,\pi]\\
f(x)&=(x_1+x_2-2)^2 + (x_1-x_2)^2 +30\min(0,(x_1-x_2))^2,\\
g(x,u)&=-x_1\cos(u)-x_2\sin(u)+1.
\end{aligned}
\end{equation*}
Like in \cite{Lasserre2011}, let $x_3 =\min(0,(x_1-x_2))$, then
$f(x) = (x_1+x_2-2)^2 + (x_1-x_2)^2 +30x_3^2$. We add new
constraints $x^2_3 = (x_1-x_2)^2$ and $x_3\geq 0$ in $X$. Replace
function $\sin(u)$ by $ u-\frac{u^3}{6}+\frac{y^5}{5!}$.
\end{exm}

\begin{exm}\label{small:example:7}
Let $U=[-1,1]$ and
\[
f(x)=x_2,\quad g(x,u)=-2x^2_1u^2+u^4-x^2_1+x_2.
\]
\end{exm}

\section{Resultants and discriminants}\label{appendix}
We review some background about \emph{resultants} and
\emph{discriminants}. More details can be found in
% Li change march 22
\cite{discriandresultant,NieDiscri,Nie2011Jacobian}.

Let $f_{1},\ldots,f_{n}$ be homogeneous polynomials in
$x=(x_1,\ldots,x_n)$. The resultant $Res(f_{1},\ldots,f_{n})$ is a
polynomial in the coefficients of $f_{1},\ldots,f_{n}$ satisfying
% Li change
\begin{equation*}\label{resultant:property}
% Li change: March 9 add "~"
Res(f_{1},\ldots,f_{n})=0~~\Leftrightarrow~~\exists~0\neq u\in
\mC^n,~f_1(u)=\cdots=f_{n}(u)=0.
\end{equation*}
Let $f_{1},\ldots,f_{m}$ be homogenous polynomials with $m< n$. The
discriminant for $f_{1},\ldots,f_{m}$ is denoted by
$\Delta(f_1,\ldots,f_m)$, which is a polynomial in the
coefficients of $f_1,\ldots,f_m$ such that
$$\Delta(f_1,\ldots,f_m)=0$$ if and only if the polynomial system
$$f_1(x)=\cdots=f_m(x)=0$$ has a solution $0\neq u\in \mC^n$ such
that the Jacobian matrix of $f_1,\ldots,f_m$ does not have full
rank.

Given inhomogeneous polynomial $h(x)\in\mR[x]$, let $\tilde{h}$ denote the
homogenization of $h$, i.e.,
$\tilde{h}=\tilde{h}(\tilde{x})=x^{\deg(h)}_0h(x/x_0)$. For
inhomogeneous polynomials $f_0, f_1, \ldots, f_n\in\mR[x]$, the resultant
$Res(f_0, f_1, \ldots, f_n)$ is defined to be
\[
Res(\tilde{f}_0, \tilde{f}_1, \ldots, \tilde{f}_{n}).
\]
For inhomogeneous polynomials $f_1,\ldots,f_m\in\mR[x]$ with $m\le n$, the
discriminant $\Delta(f_1,\ldots,f_m)$ is defined as
\[
\Delta(\tilde{f}_1,\ldots,\tilde{f}_{m}).
\]
We have
\begin{prop}\label{prop::res4inhom}
Let $f_0, f_{1},\ldots,f_{n}\in\mR[x]$ be inhomogeneous polynomials. Suppose the
polynomial system
\[
f_0(x)=f_1(x)=\cdots=f_n(x)=0
\]
has a solution in $\mC^n$, then
\[
Res(f_0,f_1,\ldots,f_n)=0.
\]
\end{prop}
\begin{proof}
If the polynomial system
\[
f_0(x)=f_1(x)=\cdots=f_n(x)=0
\]
has a solution $u\in\mC^n$, then the
polynomial system
\[
\tilde{f}_0(\tilde{x})=\tilde{f}_1(\tilde{x})=\cdots=\tilde{f}_n(\tilde{x})=0
\]
has a nonzero solution $(1,u)\in\mC^{n+1}$. The conclusion follows
by the properties of resultant for homogeneous polynomials .
\end{proof}
\begin{prop}\label{prop:equaivalent:ho} Let $m\le n$.
The polynomial system $$f_1(x)=\cdots=f_m(x)=0$$ has a solution
$u\in\mC^n$ such that the Jacobian matrix of $f_1,\ldots,f_m$ is
rank deficient at $u$ if and only if the polynomial system
$$\tilde{f}_1(\tilde{x})=\cdots=\tilde{f}_m(\tilde{x})=0$$ has a solution $(1,u)\in
\mC^{n+1}$ such that the Jacobian matrix of
$\tilde{f}_1,\ldots,\tilde{f}_m$ is rank deficient at $(1,u)$.
\end{prop}
\begin{proof}

Let $d_i=\deg_x{(f_i)}$, $f_{i,j}$ denote the homogenous part
of degree $j$ of polynomial $f_i$ and $\tilde{f}_{i,j}=x_0^{d_i-j}f_{i,j}$ for
$i=1,\cdots,m$ and $j=0,\cdots,d_i$. Denote
\[
\nabla_x:=\left\{\frac{\partial}{\partial x_1}, \cdots, \frac{\partial}{\partial
x_n}\right\}\quad\text{and}\quad \nabla_{\tilde{x}}:=\left\{\frac{\partial}{\partial x_0},
\frac{\partial}{\partial x_1}, \cdots, \frac{\partial}{\partial x_n}\right\}.
\]
The ``if'' direction is implied by
\begin{equation}\label{eq::rl}
\frac{\partial\tilde{f}_i}{\partial x_j}(1,u)=\frac{\partial f_i}{\partial
x_j}(u),\quad i=1,\cdots,m,\ j=1,\cdots,n.
\end{equation}
% Li change march 23 exists
Next we prove the ``only if'' direction. By assumption, there exists
a set of $n$ scalars $c_1,\ldots, c_n$, not all zero, such that
\[
\sum_{i=1}^m c_i(\nabla_x f_i)(u)=0
\]
which means
\[
\sum_{i=1}^mc_i \left(\sum_{j=1}^{d_i}\frac{\partial f_{i,j}}{\partial
x_k}(u)\right)=0,\quad k=1,\cdots,n.
\]
Then by Euler's Homogeneous Function Theorem, we have
\begin{equation*}
\begin{aligned}
0&=\sum_{k=1}^n\sum_{i=1}^mc_i\left(\sum_{j=1}^{d_i}\frac{\partial
f_{i,j}}{\partial x_k}(u)u_k\right)\\
&=\sum_{i=1}^mc_i\left(\sum_{j=1}^{d_i}\sum_{k=1}^n\frac{\partial
f_{i,j}}{\partial x_k}(u)u_k\right)\\
&=\sum_{i=1}^mc_i\left(\sum_{j=1}^{d_i}jf_{i,j}(u)\right)\\
&=\sum_{i=1}^mc_i\left(\sum_{j=1}^{d_i}jf_{i,j}(u)+
\sum_{j=0}^{d_i}(d_i-j)f_{i,j}(u)-\sum_{j=0}^{d_i}(d_i-j)f_{i,j}(u)\right)\\
&=\sum_{i=1}^mc_i\left(d_i\sum_{j=0}^{d_i}f_{i,j}(u)-
\sum_{j=0}^{d_i}\frac{\partial\tilde{f}_{i,j}}{x_0}(1,u)\right)\\
&=\sum_{i=1}^mc_i\left(d_if_i(u)-\frac{\partial\tilde{f}_{i}}{\partial
x_0}(1,u)\right)\\
&=-\sum_{i=1}^mc_i\frac{\partial\tilde{f}_{i}}{\partial
x_0}(1,u).
\end{aligned}
\end{equation*}
By combining (\ref{eq::rl}), we obtain
\[
\sum_{i=1}^m c_i(\nabla_{\tilde{x}} \tilde{f}_i)(1,u)=0
\]
which concludes the proof.
\end{proof}
By Proposition \ref{prop:equaivalent:ho} and the properties of
discriminant for homogeneous polynomials, we have
\begin{prop}\label{prop::dis4inhom}
Let $m\le n$ and $f_{1},\ldots,f_{m}\in\mR[x]$ be inhomogeneous
polynomials. Suppose that the polynomial system
\[
f_1(x)=\cdots=f_m(x)=0
\]
has a solution in $\mC^n$ at which the Jacobian matrix of $f_1,\ldots,f_m$
is rank deficient, then
\[
\Delta(f_1,\ldots,f_m)=0.
\]
\end{prop}
Note that the reverses of Proposition \ref{prop::res4inhom} and Proposition
\ref{prop::dis4inhom} are not necessarily true.

\def\refname{\Large\bfseries References}
\bibliographystyle{plain}
\bibliography{liwang2013}

\end{document}